\documentclass[authoryear,preprint,11pt]{elsarticle}

\usepackage{lineno,hyperref}
\usepackage[english]{babel}
\usepackage{type1ec}
\usepackage{cmap}				
\usepackage{lmodern}			
\usepackage[T1]{fontenc}		
\usepackage[utf8]{inputenc}		
\usepackage{lastpage}			
\usepackage{indentfirst}		
\usepackage{color}				
\usepackage{graphicx}			
\usepackage{natbib}
\usepackage{adjustbox}
\usepackage{amsfonts}
\usepackage{enumitem}
\usepackage{array}
\usepackage{amsmath}
\usepackage{amssymb}
\usepackage{amsthm}
\usepackage{setspace}
\usepackage{graphicx}
\usepackage{afterpage}
\usepackage[ruled,titlenumbered,lined]{algorithm2e}
\usepackage{varioref}
\usepackage{hyperref}
\usepackage{bookmark}
\usepackage{pdflscape}
\usepackage{colortbl,color}
\usepackage{subcaption}
\usepackage{soul}
\usepackage{pgf, tikz}
\usetikzlibrary{arrows, automata}
\usepackage{caption}
\usepackage{standalone}
\usepackage{longtable}

\usepackage{textcomp}
\usepackage{siunitx}
\usepackage{booktabs}
\usepackage{multirow}

\usepackage{subcaption,siunitx,booktabs}
\usepackage[margin=1in]{geometry} 

\usepackage{xcolor}
\newtheoremstyle{exampstyle}
  {\topsep} 
  {\topsep} 
  {} 
  {} 
  {\bfseries} 
  {:} 
  {.5em} 
  {} 

\allowdisplaybreaks[1]

\newtheorem{theorem}{Theorem}
\newtheorem{proposition}{Proposition}
\newtheorem{corollary}{Corollary}

\newcommand{\BigO}[1]{\ensuremath{\operatorname{O}\bigl(#1\bigr)}}

\makeatletter
\newcommand*{\rom}[1]{\expandafter\@slowromancap\romannumeral #1@}
\makeatother

\journal{arXiv.org}

\begin{document}

\begin{frontmatter}

\title{Coupling Feasibility Pump and Large Neighborhood Search to solve the Steiner team orienteering problem}



\author[dep]{Lucas Assun\c{c}\~ao\corref{cor1}}
\ead{lucas-assuncao@ufmg.br}
\author[dcc]{Geraldo Robson Mateus}
\ead{mateus@dcc.ufmg.br}
\cortext[cor1]{Corresponding author.}

\address[dep]{
  Departamento de Engenharia de Produ\c{c}\~ao, Universidade Federal de Minas Gerais\\
  Avenida Ant\^onio Carlos 6627, CEP 31270-901, Belo Horizonte, MG, Brazil\\
}

\address[dcc]{
  Departamento de Ci\^encia da Computa\c{c}\~ao, Universidade Federal de Minas Gerais\\
  Avenida Ant\^onio Carlos 6627, CEP 31270-901, Belo Horizonte, MG, Brazil\\
}

\begin{abstract}
The Steiner Team Orienteering Problem (STOP) is defined on a digraph in which arcs are associated with traverse times, and whose vertices are labeled as either \emph{mandatory} or \emph{profitable}, being the latter provided with rewards (profits). Given a homogeneous fleet of vehicles $M$, the goal is to find up to $m = |M|$ disjoint routes (from an origin vertex to a destination one) that maximize the total sum of rewards collected while satisfying a given limit on the route's duration. Naturally, all mandatory vertices must be visited. In this work, we show that solely finding a feasible solution for STOP is NP-hard and propose a Large Neighborhood Search (LNS) heuristic for the problem. The algorithm is provided with initial solutions obtained by means of the matheuristic framework known as Feasibility Pump (FP). In our implementation, FP uses as backbone a commodity-based formulation reinforced by three classes of valid inequalities. To our knowledge, two of them are also introduced in this work. The LNS heuristic itself combines classical local searches from the literature of routing problems with a long-term memory component based on Path Relinking. We use the primal bounds provided by a state-of-the-art cutting-plane algorithm from the literature to evaluate the quality of the solutions obtained by the heuristic.
Computational experiments show the efficiency and effectiveness of the proposed heuristic in solving a benchmark of 387 instances. Overall, the heuristic solutions imply an average percentage gap of only 0.54\% when compared to the bounds of the cutting-plane baseline. In particular, the heuristic reaches the best previously known bounds on 382 of the 387 instances. Additionally, in 21 of these cases, our heuristic is even able to improve over the best known bounds.
\end{abstract}


\end{frontmatter}


\section{Introduction}
\label{s_intro}

The Steiner Team Orienteering Problem (STOP) arises as a more general case of the widely studied profit collection routing problem known as the Team Orienteering Problem (TOP). 
TOP is an NP-hard problem usually defined on a complete and undirected graph, where a value of reward (profit) is associated with each vertex, and a traverse time is associated with each edge (or arc).
Given a homogeneous fleet of vehicles $M$, TOP aims at finding up to $m = |M|$ disjoint routes from an origin vertex to a destination one, such that each route has to satisfy a time constraint while the total sum of rewards collected is maximized.
In the case of STOP, the input graph is directed, and a subset of mandatory vertices is also provided. Accordingly, STOP also aims at maximizing the total sum of rewards collected within the time limit, but, now, every mandatory vertex has to be visited.

STOP finds application, for instance, in devising the itinerary of the delivery of goods performed by shipping companies \citep{Assuncao2019}. Here, a reward value --- which may rely on factors such as the urgency of the request and the customer priority --- is associated with visiting each customer. Deliveries with top priority (e.g., those whose deadlines are expiring) necessarily have to be included in the planning.
Accordingly, the goal is to select a subset of deliveries (including the top priority ones) that maximizes the total sum of rewards collected and that can be performed within a pre-established working horizon of time. 

Another application arises in the planning of home health care visits. Home health care covers a wide range of services that can be given in one's home for an illness or injury, such as monitoring a patient’s medication regimen and vital signs (blood pressure, temperature, heart rate, etc). In this case, priority could me linked to the patients' needs. Similar applications arise in the planning of other sorts of technical visits \citep{Tang05,Assuncao2019}

STOP was only introduced quite recently by \cite{Assuncao2019}. In the work, a state-of-the-art branch-and-cut algorithm from the literature of TOP is adapted to STOP, and a cutting-plane algorithm is proposed. The new algorithm relies on a compact (with a polynomial number of variables and constraints) commodity-based formulation --- also introduced in the work --- reinforced by some classes of inequalities, which consist of general connectivity constraints, classical lifted cover inequalities based on dual bounds and a class of \emph{conflict cuts}. The cutting-plane proposed not only outperforms the baseline when solving a benchmark of STOP instances, but also solves to optimality more TOP instances than any previous algorithm in the literature.

Several variations of orienteering problems have already been addressed in the literature,
combining some specific constraints, such as mandatory visits (like in STOP), multiple vehicles and time windows~\citep{Tricoire2010,Salazar2014,Archetti2014b}. Here, we discuss some of the orienteering problems that are more closely related to STOP. For instance, the single-vehicle version of STOP, known as the \emph{Steiner Orienteering Problem} (SOP), was already introduced by \cite{Letchford13}. In the work, four Integer Linear Programming (ILP) models are proposed, but no computational experiment is reported.

STOP is particularly similar to the Team Orienteering Arc Routing Problem (TOARP), in which profitable and mandatory arcs are considered (instead of vertices). \cite{Archetti2014b} introduced TOARP, along with a polyhedral study and a branch-and-cut algorithm able to solve instances with up to 100 vertices, 800 arcs and four vehicles. These instances consider sparse random and grid-based graphs. More recently, \cite{Ledesma2017} proposed a column generation approach that works by first converting the original customer-on-arc representation into a customer-on-vertex one, thus creating STOP-like instances. Their results indicate that the column generation algorithm is a complementary approach to the branch-and-cut of \cite{Archetti2014b}, being useful in those cases where the latter shows poor performance. 

To the best of our knowledge, only two heuristics have been proposed for TOARP: the matheuristic of \cite{Archetti2015} and the iterated local search heuristic of \cite{Liangjun2017}.
With respect to the single-vehicle version of TOARP, known as the Orienteering Arc Routing Problem (OARP), \cite{Archetti2016} developed of a branch-and-cut algorithm based on families of facet-inducing inequalities also introduced in the work.

The specific case of STOP with no mandatory vertices --- namely TOP --- has drawn significant attention in the past decades, and several exact and heuristic algorithms have been proposed to solve the problem, as summarized next.

The previous exact algorithms for TOP are based on cutting planes and column generation. In summary, they consist of branch-and-cut \citep{Dang13b,Bianchessi2017,Hanafi2020}, branch-and-price \citep{Boussier07,Keshtkaran16}, branch-and-cut-and-price \citep{Poggi10,Keshtkaran16} and cutting-plane algorithms \citep{ElHajj2016,Assuncao2019}. Among them, we highlight the branch-and-cut of \cite{Hanafi2020} as the current state-of-the-art for TOP, being able to solve to optimality five more instances than the previous state-of-the-art \citep{Assuncao2019}.

Regarding heuristics for TOP, several ones have been proposed for TOP, adapting a wide range of metaheuristic frameworks, such as \emph{Tabu Search} \citep{Tang05}, \emph{Variable Neighborhood Search} (VNS) \citep{Archetti07,Vansteenwegen09}, \emph{Ant Colony} \citep{Ke08}, \emph{Greedy Randomized Adaptive Search Procedure} (GRASP) \citep{Souffriau10}, \emph{Simulated Annealing} \citep{Lin13}, \emph{Large Neighborhood Search} (LNS) \citep{Kim13,Vidal2015} and \emph{Particle Swarm Optimization} (PSO) \citep{Dang13}.
To our knowledge, the latest heuristic for TOP was proposed by \cite{Ke2016}.
Their heuristic, namely \emph{Pareto mimic algorithm}, introduces a so-called \emph{mimic operator} to generate new solutions by imitating incumbent ones. The algorithm also adopts the concept of \emph{Pareto dominance} to update the population of incumbent solutions by considering multiple indicators that measure the quality of each solution.
We refer to \cite{Vansteenwegen2011} and \cite{Gunawan2016} for detailed surveys on exact and heuristic approaches for TOP and some of its variants.

Although the heuristics proposed for TOP follow different metaheuristic frameworks, most of them have two key elements in common: the way they generate initial solutions and the local search operators used to improve the quality of the solutions, which usually consist of the classical 2-opt and its variations~\citep{Lin1965}, coupled with vertex shiftings and exchanges between routes. While these local searches naturally apply to STOP, finding a feasible solution for STOP is not always an easy task. 
In TOP, feasible solutions are usually built by iteratively adding vertices to initially empty routes according to greedy criteria \citep{Chao96,Archetti07,Vansteenwegen09,Souffriau10,Lin13,Kim13}. On the other hand, in STOP, a trivial solution with empty routes is not necessarily feasible, since mandatory vertices have to be visited.
In fact, as we formally prove in the remainder of this work, solely finding a feasible solution for an STOP instance (or proving its infeasibility) is NP-hard in the general case.

In this work, we propose an LNS heuristic for STOP. As far as we are aware, this is also the first heuristic devised to solve the general case of STOP. The LNS heuristic combines some of the aforementioned classical local searches from the literature of TOP and other routing problems with a long-term memory component based on Path Relinking (PR)~\citep{Glover1997}. The proposed heuristic tackles the issue of finding an initial solution by applying the Feasibility Pump (FP) matheuristic framework of \cite{Fischetti2005}. In particular, FP has been successfully used to find initial feasible solutions for several challenging problems modelled by means of Mixed Integer Linear Programming (MILP)~\citep{Berthold2019}.

Given an MILP formulation for the problem addressed, FP works by sequentially and wisely rounding fractional solutions obtained from solving auxiliary problems based on the linear relaxation of the original MILP formulation. In this work, we take advantage from the compactness of the formulation presented by \cite{Assuncao2019} to use it within the FP framework. In our experiments, we also consider a reinforced version of this formulation obtained from the addition of three classes of valid inequalities. The idea is that a stronger model, being closer to the convex hull, might help FP converging to an integer solution within less iterations. In fact, as later endorsed by our experiments, the strengthening of the original formulation pays off. 

The three classes of valid inequalities adopted in this work consist of \emph{clique conflict cuts}, \emph{arc-vertex inference cuts} and classical lifted cover inequalities based on dual bounds. The latter one was already used in the cutting-plane of \cite{Assuncao2019}, but, as far as we are aware, the first two classes are new valid inequalities for the problem. In particular, the clique conflict cuts generalize and extend the general connectivity constraints and cuts based on vertex conflicts used in previous works \citep{ElHajj2016,Bianchessi2017,Assuncao2019}. The due proofs regarding the validity of these new inequalities and the dominance of clique conflict cuts over previously adopted inequalities are also given.

In summary, the contributions of this work are threefold: (i) the problem of finding a feasible solution for STOP is proven to be NP-hard, (ii) a previously introduced formulation for STOP is further reinforced by the addition of three classes of valid inequalities. As far as we are aware, two of them --- namely clique conflict cuts and arc-vertex inference cuts --- are also introduced in this work. Moreover, (iii) an LNS heuristic that uses FP to find initial solutions is presented and experimentally evaluated. To this end, we consider the primal bounds provided by the cutting-plane algorithm of \cite{Assuncao2019} as a baseline. Our computational results show the efficiency and effectiveness of the proposed heuristic in solving a benchmark of STOP instances adapted from the literature of TOP. Overall, the heuristic solutions imply an average percentage gap of only 0.54\% when compared to the bounds of the cutting-plane baseline. In particular, the heuristic reaches the best previously known bounds on 382 out of the 387 instances considered. Additionally, in 21 of these cases, the heuristic is even able to improve over the best known bounds.

The remainder of this work is organized as follows. STOP is formally defined in Section~\ref{s_notation}, along with the notation adopted throughout the work.
In Section~\ref{s_np-hard_proof}, we prove that solely finding a feasible solution for STOP is NP-hard. In Section~\ref{s_models}, we describe an MILP formulation for STOP, as well as three families of inequalities used to reinforce it. In the same section, we give the proofs of validity for the inequalities (when the case) and describe the procedures used to separate them.
In Section~\ref{s_cutting-plane}, we present the general cutting-plane procedure used to reinforce the original STOP formulation by the addition of the three classes of inequalities discussed in this work.
Sections~\ref{s_FP} and~\ref{s_LNS} are devoted to detailing the FP procedure used to generate initial solutions and the LNS heuristic proposed, respectively.
Some implementation details are given in Section~\ref{s_implementation_details}, followed by computational results (Section~\ref{s_experiments}). Concluding remarks are provided in the last section. 

\section{Problem definition and notation}
\label{s_notation}
STOP is defined on a (not necessarily complete) digraph $G=(N,A)$, where $N$ is the vertex set, and $A$ is the arc set. Let $s,t \in N$ be the origin and the destination vertices, respectively, with $s \neq t$. Moreover, let $S \subseteq N\backslash\{s,t\}$ be the subset of \emph{mandatory}
vertices, and $P \subseteq N\backslash\{s,t\}$ be the set of \emph{profitable} vertices, 
such that $N = S \cup P \cup \{s,t\}$ and $S \cap P = \emptyset$. A reward $p_i \in \mathbb{Z}^+$ is associated with each vertex $i \in P$, and a traverse time $d_{ij} \in \mathbb{R}^+$ is associated with each arc $(i,j) \in A$. Additionally, a fleet $M$ of homogeneous vehicles is available to visit the vertices, and each vehicle can run for no more than a time limit $T$. For short, an STOP instance is defined as $\langle G,S,P,s,t,p,d,M,T \rangle$.

The goal is to find up to $m = |M|$ routes from $s$ to $t$ --- one for each vehicle used --- such that every mandatory vertex in $S$ belongs to exactly one route and the total sum of rewards collected by visiting profitable vertices is maximized. Each profitable vertex in $P$ can be visited by at most one vehicle, thus avoiding the multiple collection of a same reward.

Given a subset $V \subset N$, we define the sets of arcs leaving and entering $V$ as $\delta^+(V) = \{(i,j)\in A:\, i \in V,\, j \in N \backslash V\}$ and $\delta^-(V) = \{(i,j)\in A:\, i \in N\backslash V,\, j \in V\}$, respectively. Similarly, given a vertex $i \in N$,
we define the sets of vertices $\delta^+(i) = \{j\in N:\, (i,j) \in A\}$ and $\delta^-(i) = \{j\in N:\, (j,i)\in A\}$.
Moreover, given two arbitrary vertices $i,j \in N$ and a path $p$ from $i$ to $j$ in $G$, we define $A_p \subseteq A$ as the arc set of $p$. 

\section{Building a feasible solution for STOP is NP-hard}
\label{s_np-hard_proof}
In TOP, a trivial feasible solution consists of a set of empty routes, one for each vehicle. This solution can be easily improved by inserting vertices while not exceeding the routes' duration limit. In fact, the heuristics in the literature of TOP make use of this simple procedure, usually adopting greedy criteria to iteratively add vertices to the empty routes \citep{Chao96,Archetti07,Vansteenwegen09,Souffriau10,Lin13,Kim13}. Once we consider mandatory vertices (as in STOP), a trivial solution with empty routes is no longer feasible.
Formally, consider the optimization problem of finding (building) a feasible solution for an arbitrary STOP instance, namely Feasibility STOP (FSTOP), defined as follows.
\begin{description}[leftmargin=2cm]
\item[(FSTOP)]
\begin{description}[leftmargin=1.45cm]
\item[Input:] An STOP instance $\langle G,S,P,s,t,p,d,M,T \rangle$.
\item[Output:] A feasible solution (if any) for the input STOP instance.
\end{description}
\end{description}

\noindent Notice that FSTOP is not interested in deciding if a given STOP instance is feasible, but in actually determining a feasibility certificate (i.e., an actual solution) when the case. 

In this section, we show that FSTOP is NP-hard by proving that its decision version, namely Decision FSTOP (D-FSTOP), is already NP-hard.
To this end, we reduce the Hamiltonian Path Problem (HPP) --- which is known to be NP-complete \citep{Garey79} --- to FSTOP.
\begin{description}[leftmargin=2.5cm]
\item[(D-FSTOP)]
\begin{description}[leftmargin=1.45cm]
\item[Input:] An STOP instance $\langle G,S,P,s,t,p,d,M,T \rangle$.
\item[Question:] Is there a feasible solution for the input STOP instance?
\end{description}
\end{description}

\begin{description}[leftmargin=1.45cm]
\item[(HPP)]
\begin{description}[leftmargin=1.45cm]
\item[Input:] A digraph $G=(N,A)$, where $N$ is the vertex set, and $A$ is the arc set. An origin vertex $s \in N$ and a destination vertex $t \in N$. For short, $\langle G,s,t \rangle$.
\item[Question:] Is there an Hamiltonian path from $s$ to $t$ in $G$?
\end{description}
\end{description}

\begin{theorem}
\label{teo_reduction}
HPP $\leq_p$ D-FSTOP, i.e., HPP is polynomial-time reducible to D-FSTOP.
\end{theorem}
\begin{proof}
First, we propose a straightforward polynomial time reduction of HPP instances to D-FSTOP ones. Given an HPP instance $\langle G,s,t \rangle$, we build a corresponding D-FSTOP instance (which is also an STOP instance) by considering the same digraph $G$, where $s$ and $t$ are its origin and destination vertices as well. A single vehicle is considered, and all the vertices of $G$ are set as mandatory (except for $s$ and $t$). Accordingly, the set of profitable vertices is empty, and the traverse time vector is set to $d = \textbf{1}$. The time limit $T$ is a sufficiently large number (say, $T \geq |N|-1$), as to allow the selection of all mandatory vertices in a route.

Now, we only have to show that an HPP instance has an Hamiltonian path from $s$ to $t$ if, and only if, the corresponding D-FSTOP instance (built as described above) gives a ``yes'' answer. The validity of this proposition is quite intuitive, as (i) the pair of HPP and D-FSTOP instances shares the same graph, (ii) the time limit $T$ allows that all vertices belong to an STOP solution, and (iii) by definition, the single STOP route must visit each vertex exactly once, which defines an Hamiltonian path.
\end{proof}
%

\begin{corollary}
\label{corol_np_hard_feasibility}
D-FSTOP and, thus, FSTOP, are NP-hard.
\end{corollary}

As a consequence of Corollary~\ref{corol_np_hard_feasibility}, heuristics for STOP are subject to an additional challenge, which consists in building an initial feasible solution in reasonable time. To address this issue, we make use of the FP matheuristic of \cite{Fischetti2005}, which is described in Section~\ref{s_FP}. As a matheuristic, FP uses a mathematical formulation to guide its search for a feasible solution. Then, before describing FP, we provide an MILP formulation for STOP and discuss some families of inequalities used to further strengthen it.

\section{Mathematical foundation}
\label{s_models}
In the MILP formulation proposed by \cite{Assuncao2019} for STOP --- hereafter denoted by $\mathcal{F}$ --- time is treated as a commodity to be spent by the vehicles when traversing each arc in their routes, such that every vehicle departs from $s$ with an initial amount of $T$ units of commodity, the time limit.

Before defining $\mathcal{F}$, let $R_{ij}$ denote the minimum time needed to reach a vertex $j$ when departing from a vertex $i$ in the graph $G$ of a given STOP instance, i.e., $R_{ij} = \min\{\sum\limits_{a \in A_p}{d_a}:\,\mbox{$p$ is a path from $i$ to $j$ in $G$}\}$. Accordingly, $R_{ii}$ = 0 for all $i \in N$, and, if no path exists from a vertex to another, the corresponding entry of $R$ is set to infinity. One may observe that $R$ is not necessarily symmetric, since ${G}$ is directed. Moreover, considering that the traverse times associated with the arcs of ${G}$ are non-negative (and, thus, no negative cycle exists), this $R$ matrix can be computed \emph{a priori} (for each instance) by means of the classical dynamic programming algorithm of Floyd-Warshall \citep{Cormen2001}, for instance.

We introduce the following sets of variables.
The binary variables $y$ and $x$ identify the solution routes, such that $y_{i} = 1$ if the vertex $i \in N$ is visited by a vehicle
of the fleet ($y_{i} = 0$, otherwise), and $x_{ij} = 1$ if the arc $(i,j) \in A$ is traversed in the solution ($x_{ij} = 0$, otherwise).
In addition, the continuous flow variables $f_{ij}$, for all $(i,j) \in A$, represent the amount of time still available for a vehicle after traversing the arc $(i,j)$ as not to exceed $T$. The slack variable $\varphi$ represents the number of vehicles that are not used in the solution.
$\mathcal{F}$ is defined from (\ref{csc100}) to (\ref{csc114}).

\begin{eqnarray}
  \mbox{($\mathcal{F}$)\quad}\max && \sum \limits_{i \in P}{p_{i}y_{i}}, \label{csc100}\\
  s.t. && y_i = 1 \qquad \forall\, i \in S\cup\{s,t\}, \label{b101} \\
	&& \sum \limits_{j \in \delta^{+}(i)}{x_{ij}} = y_i \qquad \forall\, i \in S\cup P, \label{b102} \\  
	&& \sum \limits_{j \in \delta^{+}(s)}{x_{sj}} = \sum \limits_{i \in \delta^{-}(t)}{x_{it}} = m - \varphi, \label{b103} \\
	&& \sum \limits_{i \in \delta^{-}(s)}{x_{is}} = \sum \limits_{j \in \delta^{+}(t)}{x_{tj}} = 0, \label{b104} \\
	&& \sum \limits_{j \in \delta^{+}(i)}{x_{ij}} - \sum \limits_{j \in \delta^{-}(i)}{x_{ji}} = 0 \qquad \forall\, i \in S\cup P,\label{b105} \\
	&& f_{sj} = (T-d_{sj})x_{sj} \qquad \forall\, j \in \delta^+(s), \label{csc107}\\
	&& \sum \limits_{j \in \delta^{-}(i)}{f_{ji}} - \sum \limits_{j \in \delta^{+}(i)}{f_{ij}} = \sum \limits_{j \in \delta^{+}(i)}{d_{ij}x_{ij}}\qquad \forall\, i \in S\cup P,\label{csc108} \\
	&& f_{ij} \leq (T- R_{si} - d_{ij})x_{ij} \qquad \forall (i,j) \in A,\, i \neq s, \label{csc109} \\
	&& f_{ij} \geq R_{jt}x_{ij} \qquad \forall (i,j) \in A, \label{csc110} \\
    && x_{ij} \in \{0,1\} \qquad \forall (i,j) \in A, \label{csc111} \\
    && y_{i} \in \{0,1\} \qquad \forall i\, \in N, \label{csc112} \\
    && f_{ij} \geq 0 \qquad \forall (i,j) \in A, \label{csc113} \\
    && 0 \leq \varphi \leq m. \label{csc114}
\end{eqnarray}

The objective function in (\ref{csc100}) gives the total reward collected by visiting profitable vertices.
Constraints (\ref{b101}) impose that all mandatory vertices (as well as $s$ and $t$) are selected, while constraints (\ref{b102}) ensure that each vertex in $S \cup P$ is visited at most once. Restrictions (\ref{b103}) ensure that at most $m$ vehicles
leave the origin $s$ and arrive at the destination $t$, whereas constraints (\ref{b104}) impose that vehicles cannot arrive at
$s$ nor leave $t$. Moreover, constraints (\ref{b105}), along with constraints (\ref{b101}) and (\ref{b102}), guarantee that,
if a vehicle visits a vertex $i \in S \cup P$, then it must enter and leave this vertex exactly once.
Constraints (\ref{csc107})-(\ref{csc109}) ensure that each of the solution routes has a total traverse time of at most $T$. Precisely,
restrictions (\ref{csc107}) implicitly state, along with (\ref{b103}), that the total flow available at the origin $s$ is $(m-\varphi)T$, and, in particular, each vehicle (used) has an initial amount of $T$
units of flow. Constraints (\ref{csc108}) manage the flow consumption incurred from traversing the arcs selected, whereas constraints (\ref{csc109}) impose that an arc $(i,j) \in A$ can only be traversed if the minimum time of a route from $s$ to $j$ through $(i,j)$ does not exceed $T$. In (\ref{csc109}), we do not consider the arcs leaving the origin, as they are already addressed by (\ref{csc107}).
Restrictions (\ref{csc110}) are, in fact, valid inequalities that give lower bounds on the flow passing through each arc, and constraints (\ref{csc111})-(\ref{csc114}) define the domain of the variables. Here, the management of the flow associated with the variables $f$ also avoids the existence of sub-tours in the solutions.

Although the $y$ variables can be easily discarded --- as they solely aggregate specific subsets of the $x$ variables --- they enable us to represent some families of valid inequalities (as detailed in Section~\ref{s_cuts}) by means of less dense cuts. For this reason, they are preserved. Also notice that the dimension of formulation $\mathcal{F}$ (both in terms of variables and constraints) is $\BigO{|A|+|N|}$, being completely independent from the number of vehicles available.

\subsection{Families of valid inequalities}
\label{s_cuts}
In this section, we discuss three families of valid inequalities able to reinforce $\mathcal{F}$. They consist of clique conflict cuts, a class of the so-called arc-vertex inference cuts and classical lifted cover inequalities based on dual bounds. 
To the best of our knowledge, the first and second classes of inequalities are also introduced in this work. After the description of each class of inequality, we detail their corresponding separation procedures, which are called within the cutting-plane algorithm later discussed in Section~\ref{s_cutting-plane}. Hereafter, the linear relaxation of $\mathcal{F}$ is denoted by $\mathcal{L}$.

\subsubsection{Clique Conflict Cuts (CCCs)}
Consider the set $\mathcal{K}$ of vertex pairs which cannot be simultaneously in a same valid route. Precisely, for every pair $\{ i,j\} \in \mathcal{K}$, with $i,j \in N\backslash\{s,t\}$, we have that any route from $s$ to $t$ that visits $i$ and $j$ (in any order) has a total traverse time that exceeds the limit $T$.
Accordingly, consider the undirected conflict graph $G_c = (N\backslash\{s,t\},\mathcal{K})$. 
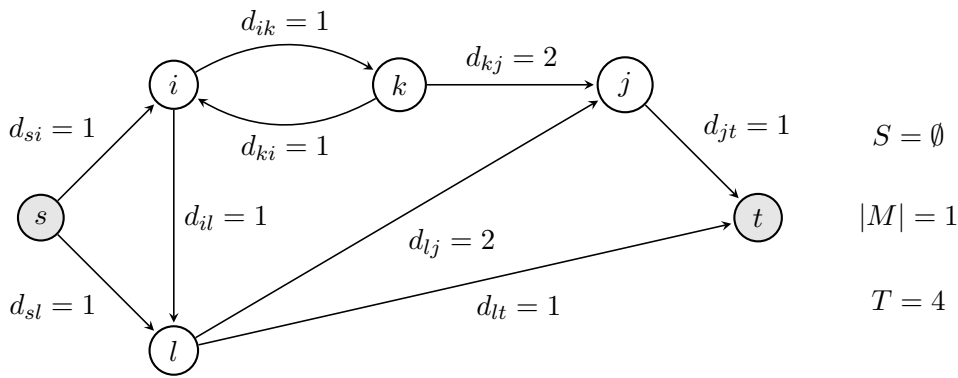
\begin{figure}[!ht]
\centering
  \begin{tikzpicture}[
            > = stealth, 
            shorten > = 1pt, 
            auto,
            node distance = 2.5cm, 
            semithick 
        ]

        \tikzstyle{every state}=[
            draw = black,
            thick,
            fill = white,
            minimum size = 4mm
        ]

        \node[state] (s) [fill = gray!20]{$s$};
        \node[state] (i) [above right of=s] {$i$};
        \node[state] (l) [below right of=s] {$l$};
        \node[state] (k) [right of=i,node distance = 3cm] {$k$};
        \node[state] (j) [right of=k,node distance = 3cm] {$j$};
        \node[state] (t) [below right of=j,fill = gray!20] {$t$};
        
        
        \node[state] (M) [right of = t, draw = none, node distance = 2cm]{$|M| = 1$};
        \node[state] (S) [above of = M, fill= none, draw = none, node distance = 1.1cm]{$S = \emptyset$};
         \node[state] (T) [below of = M, draw = none, node distance = 1.1cm]{$T = 4$};
    
        \path[->] (s) edge node [above left] {$d_{si}= 1$} (i);
        \path[->] (s) edge node [below left] {$d_{sl}= 1$} (l);
        \path[->] (l) edge node [below right] {$d_{lj}= 2$} (j);
        \path[->] (i) edge [bend left] node [above] {$d_{ik}= 1$} (k);
        \path[->] (k) edge [bend left] node [below] {$d_{ki}= 1$} (i);
        \path[->] (k) edge node [above] {$d_{kj}= 2$} (j);
        \path[->] (j) edge node [above right] {$d_{jt}= 1$} (t);
        \path[->] (l) edge node [below right] {$d_{lt}= 1$} (t);
        \path[->] (i) edge node {$d_{il}= 1$} (l);
    
    \end{tikzpicture}
\caption{An example of an STOP instance. Profit values are omitted.}
\label{fig:ex}
\end{figure}

Consider the STOP instance defined by the digraph shown in Figure~\ref{fig:ex}, whose arcs have traverse times $d_{si} = d_{sl} = d_{il} = d_{ik} = d_{ki} = d_{lt} = d_{jt} = 1$ and $d_{kj} = d_{lj} = 2$. In this case, $S = \emptyset$ (i.e., there is no mandatory vertex), and a single vehicle is available to move from $s$ to $t$, with $T= 4$. The conflict graph related to this instance is given in Figure~\ref{fig:ex:conflictgraph}.

\begin{figure}[!ht]
\centering
\begin{tikzpicture}[
            > = stealth, 
            auto,
            node distance = 1.5cm, 
            semithick 
        ]

\def \n {6}
\def \radius {1.7cm}
\def \margin {10} 

        \tikzstyle{every state}=[
            draw = black,
            thick,
            fill = white,
            minimum size = 4mm
        ]

        \node[state] (i) {$i$};
        \node[state] (l) [below of=i] {$l$};
        \node[state] (k) [right of=i,node distance = 3cm] {$k$};
        \node[state] (j) [right of=k,node distance = 3cm] {$j$};
\path[-] (i) edge node {} (k);
\path[-] (j) edge [bend right] node {} (i);
\path[-] (j) edge node {} (k);
\path[-] (k) edge node {} (l);

\end{tikzpicture}
\caption{Conflict graph related to the STOP instance of Figure~\ref{fig:ex}.}
\label{fig:ex:conflictgraph}
\end{figure}
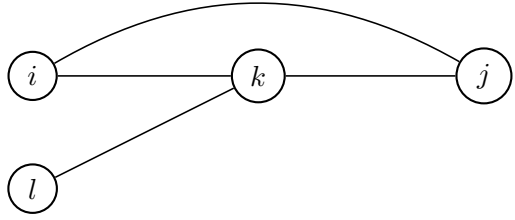

Now, let $\Sigma$ be the set of all the cliques of $G_c$, which are referred to as \emph{conflict cliques}. Then, CCCs are defined as
\begin{eqnarray}
\sum\limits_{e \in \delta^-(V)}{x_{e}} \geq \sum\limits_{i \in \sigma}y_{i} \qquad \forall\; \sigma \in \Sigma,\; \forall\; V \subseteq N\backslash\{s\},\, \sigma \subseteq V, \label{ccc00}\\
\sum\limits_{e \in \delta^+(V)}{x_{e}} \geq \sum\limits_{i \in \sigma}y_{i}  \qquad \forall\; \sigma \in \Sigma,\; \forall\; V \subseteq N\backslash\{t\},\, \sigma \subseteq V.\label{ccc01}
\end{eqnarray}

\noindent Intuitively speaking, CCCs state that the number of visited vertices from any conflict clique gives a lower bound on the number of routes needed to make a solution feasible for STOP. In other words, vertices of a conflict clique must necessarily belong to different routes in any feasible solution for $\mathcal{F}$. The formal proof of the validity of CCCs is given as follows.

\begin{proposition}
\label{prop_cccs}
Inequalities (\ref{ccc00}) do not cut off any feasible solution of $\mathcal{F}$.
\end{proposition}
\begin{proof}
Consider an arbitrary feasible solution $(\bar{x},\bar{y},\bar{f},\bar{\varphi})$ for $\mathcal{F}$, a conflict clique $\sigma \in \Sigma$ and a subset $V \subseteq N \backslash \{s\}$, with $\sigma \subseteq V$. Then, we have two possibilities:
\begin{enumerate}
\item if $\sum\limits_{i \in \sigma}{\bar y}_{i} = 0$, then $\overbrace{\sum\limits_{e \in \delta^-(V)}{x_{e}}}^{\geq\,0} \geq \overbrace{\sum\limits_{i \in \sigma}y_{i}}^{=\,0}$.
\item if $\sum\limits_{i \in \sigma}{\bar y}_{i} > 0$, then, since $s \notin V$, there must be a subset of arcs $A' \subseteq \delta^-(V)$, with $\bar{x}_{e} = 1$ for all $e \in A'$, and $|A'| = \sum\limits_{i \in \sigma}{\bar y}_{i}$, so that every vertex $i \in \sigma$ having ${\bar y}_i =1$ is traversed in a different route from $s$. Then, also in this case, $\overbrace{\sum\limits_{e \in \delta^-(V)}{x_{e}}}^{\geq\,|A'|} \geq \overbrace{\sum\limits_{i \in \sigma}y_{i}}^{=\,|A'|}$. \qedhere
\end{enumerate}
\end{proof}

\begin{corollary}
\label{corol02}
Inequalities (\ref{ccc01}) do not cut off any feasible solution of $\mathcal{F}$.
\end{corollary}
\begin{proof}
The same mathematical argumentation of Proposition~\ref{prop_cccs} can be used to prove the validity of inequalities (\ref{ccc01}) by simply replacing $s$ and $\delta^-(V)$ with $t$ and $\delta^+(V)$, respectively.  
\end{proof}

Notice that, from (\ref{b105}), we have that $\sum\limits_{e \in \delta^-(V)}{x_{e}} = \sum\limits_{e \in \delta^+(V)}{x_{e}}$ for all $V \subseteq S \cup P = N\backslash\{s,t\}$. Therefore, for all $V \subseteq N\backslash\{s,t\}$, inequalities (\ref{ccc00}) and (\ref{ccc01}) cut off the exact same regions of the polyhedron $\mathcal{L}$, the linearly relaxed version of $\mathcal{F}$.. In this sense, the whole set of CCCs can be represented in a more compact manner
by (\ref{ccc00}) and 
\begin{eqnarray}
\sum\limits_{e \in \delta^+(V)}{x_{e}} \geq \sum\limits_{i \in \sigma}y_{i} \qquad \forall\; \sigma \in \Sigma,\; \forall\; V \subseteq N \backslash \{t\},\, \sigma \cup \{s\} \subseteq V.\label{ccc02}
\end{eqnarray}

\begin{theorem}
\label{teo_ccc_dominance}
For every non-maximal conflict clique $\sigma_1 \in \Sigma$, there is a clique $\sigma_2 \in \Sigma$, with $\sigma_1 \subset \sigma_2$, such that the CCCs referring to $\sigma_1$ are dominated by the ones referring to $\sigma_2$, when considering the polyhedron $\mathcal{L}$.
\end{theorem}
\begin{proof}
Consider an arbitrary non-maximal conflict clique $\sigma_1 \in \Sigma$ and, without loss of generality, a solution $(\bar{x},\bar{y},\bar{f},\bar{\varphi})$ for $\mathcal{L}$ that is violated by at least one of the CCCs referring to $\sigma_1$. We have to show that there exists a clique $\sigma_2 \in \Sigma$, with $\sigma_1 \subset \sigma_2$, whose corresponding CCCs also cut off $(\bar{x},\bar{y},\bar{f},\bar{\varphi})$ from $\mathcal{L}$. To this end, consider a vertex $k \in (S \cup P) \backslash \sigma_1$, such that $\{k,j\} \in \mathcal{K}$ for all $j \in \sigma_1$. Such vertex exists, since $\sigma_1$ is supposed to be non-maximal. Then, define the conflict clique $\sigma_2 = \sigma_1 \cup \{k\}$.

By assumption, $(\bar{x},\bar{y},\bar{f},\bar{\varphi})$ violates at least one of the CCCs (\ref{ccc00}) and (\ref{ccc01}) referring to $\sigma_1$. Then, we must have at least one of these two possibilities:
\begin{enumerate}
\item $\exists\, V' \subseteq N\backslash\{s\}$, with $\sigma_1 \subseteq V'$, such that $\sum\limits_{e \in \delta^-(V')}{{\bar x}_{e}} < \sum\limits_{i \in \sigma_1}{\bar y}_{i}$. 

Define the set $V'' = V'\cup\{k\}$. Notice that, if we originally have $k \in V'$, then the CCC (\ref{ccc00}) considering $\sigma_2$ and $V=V'(=V''$ in this case) is also violated by $(\bar{x},\bar{y},\bar{f},\bar{\varphi})$.
Otherwise, if $k \notin V'$, then
\begin{align}
   \sum\limits_{e \in \delta^-(V'')}{{\bar x}_{e}}= \sum\limits_{e \in \delta^-(V')}{{\bar x}_{e}} + \overbrace{\sum\limits_{i \in \delta^-(k)\backslash V'}{{\bar x}_{ik}}}^{(g)} - \overbrace{\sum\limits_{j \in \delta^+(k)\cap V'}{{\bar x}_{kj}}}^{(h)} \label{proof_dominance00},
\end{align}
\noindent which follows from the fact that the difference between the summation of arcs entering $V''$ and that of arcs entering $V'$ corresponds to the difference between (g) the sum of arcs arriving at $k$ that do not leave a vertex of $V'$ and (h) the sum of arcs leaving $k$ that arrive at a vertex of $V'$. In other words, (g) considers the arcs that traverse the cut $(N\backslash V'',V'')$ but do not traverse $(N\backslash V',V')$, while (h) considers the arcs that do not traverse the cut $(N\backslash V'',V'')$ but traverse $(N\backslash V',V')$.

From (\ref{proof_dominance00}) and the hypothesis that $\sum\limits_{e \in \delta^-(V')}{{\bar x}_{e}} < \sum\limits_{i \in \sigma_1}{\bar y}_{i}$, it follows that
\begin{align}
   \sum\limits_{e \in \delta^-(V'')}{{\bar x}_{e}} < \sum\limits_{i \in \sigma_1}{\bar y}_{i} + \overbrace{\sum\limits_{i \in \delta^-(k)\backslash V'}{{\bar x}_{ik}}}^{(g)} - \overbrace{\sum\limits_{j \in \delta^+(k)\cap V'}{{\bar x}_{kj}}}^{(h)}. \label{proof_dominance01}
\end{align}
\noindent From (\ref{b102}), (\ref{b105}), (\ref{csc111}) and (\ref{csc112}), we have that 
\begin{align}
    \overbrace{\sum\limits_{i \in \delta^-(k)\backslash V'}{{\bar x}_{ik}}}^{(g)} - \overbrace{\sum\limits_{j \in \delta^+(k)\cap V'}{{\bar x}_{kj}}}^{(h)} \leq {\bar y}_k.
\end{align}
\noindent Then, from (\ref{proof_dominance01}), it follows that 
\begin{align}
   \sum\limits_{e \in \delta^-(V'')}{{\bar x}_{e}} & < \sum\limits_{i \in \sigma_1}{\bar y}_{i} + {\bar y}_k \\
   & < \sum\limits_{i \in \sigma_2}{\bar y}_{i},
\end{align}

Since $\sigma_2 = \sigma_1 \cup \{k\}$ and $\sigma_1 \subseteq V'$, we have that $\sigma_2 \subseteq V''$. Then, also in this case, there exists a CCC (\ref{ccc00}) referring to $\sigma_2$ (in particular, with $V = V''$) that is also violated by $(\bar{x},\bar{y},\bar{f},\bar{\varphi})$.

\item $\exists\, V' \subseteq N\backslash\{t\}$, with $\sigma_1 \subseteq V'$, such that $\sum\limits_{e \in \delta^+(V')}{{\bar x}_{e}} < \sum\limits_{i \in \sigma_1}{\bar y}_{i}$.

Through the same idea of the previous case, we can also show that the CCC (\ref{ccc01}) referring to $\sigma_2$ that considers $V = V' \cup \{k\}$ is also violated by $(\bar{x},\bar{y},\bar{f},\bar{\varphi})$.
\end{enumerate} \qedhere \end{proof}

From Theorem~\ref{teo_ccc_dominance}, we may discard several CCCs (\ref{ccc00}) and (\ref{ccc01}), as we only need to consider the ones related to maximal conflict cliques. We highlight that, in the conflict graph, there might be maximal cliques of size one, which correspond to isolated vertices.
We also remark that CCCs are a natural extension of two classes of valid inequalities previously addressed in the literature of TOP and STOP, namely general connectivity constraints~\citep{Bianchessi2017} and conflict cuts~\citep{Assuncao2019}. In particular, these two classes consist of CCCs based on cliques of sizes one and two, respectively. Then, the result below is another direct implication of Theorem~\ref{teo_ccc_dominance}.

\begin{corollary}
\label{corol_dominance}
CCCs dominate the general connectivity constraints of \cite{Bianchessi2017} and the conflict cuts of \cite{Assuncao2019}.
\end{corollary}

\subsubsection{Separation of CCCs}
\label{s_sep_cccs}
Let $(\bar{x},\bar{y},\bar{f},\bar{\varphi})$ be a given fractional solution referring to $\mathcal{L}$, the linear relaxation of $\mathcal{F}$. Also consider the residual graph $\tilde{G} = (N,\tilde{A})$ induced by $(\bar{x},\bar{y},\bar{f},\bar{\varphi})$, such that each arc $(i,j) \in A$ belongs to $\tilde{A}$ if, and only if, $\bar{x}_{ij} > 0$. Moreover, a capacity $c[i,j] = \bar{x}_{ij}$ is associated with each arc $(i,j) \in \tilde{A}$.

As detailed in the sequel, the separation of CCCs involves solving maximum flow problems. In this sense, consider the notation defined as follows.
Given an arbitrary digraph ${G}_a$ with capacitated arcs, and two vertices $i$ and $j$ of ${G}_a$, let ${\textit{max-flow}}_{i\to j}{({G}_a)}$ denote the problem of finding the maximum flow (and, thus, a minimum cut) from $i$ to $j$ on ${G}_a$.
Moreover, let $\{ F_{{i\to j}},\theta_{i\to j} \}$ denote an optimal solution of such problem, where $F_{{i\to j}}$ is the value of the maximum flow, and $\theta_{i\to j}$ defines a corresponding minimum cut, with $i \in \theta_{i\to j}$. 

Considering the shortest (minimum time) paths matrix $R$ defined in the beginning of Section~\ref{s_models}, we first compute the set $\mathcal{K}$ of conflicting vertices by checking, for all pairs $\{ i,j\}$, $i,j \in N \backslash \{s,t\}$, $i \neq j$, if there exists a path from $s$ to $t$ on ${G}$ that traverses both $i$ and $j$ (in any order) and that satisfies the total time limit $T$. If no such path exists, then $\{ i,j\}$ belongs to $\mathcal{K}$. For simplicity, in this work, we only consider a subset $\tilde{\mathcal{K}} \subseteq \mathcal{K}$ of conflicting vertex pairs, such that
\begin{equation*}
\{ i,j\}\in \tilde{\mathcal{K}} \text{ iff }
\begin{cases}
        \quad (i)\; R_{si} + R_{ij} + R_{jt} > T,\; \text{and} \\
        \quad (ii)\; R_{sj} + R_{ji} + R_{it} > T, 
\end{cases}
\qquad{ \forall\, i,j \in N \backslash \{s,t\}, i \neq j.}
\end{equation*}

\noindent where $(i)$ is satisfied if a minimum traverse time route from $s$ to $t$ that visits $i$ before $j$ exceeds the time limit. Likewise, $(ii)$ considers a minimum time route that visits $j$ before $i$.
Since the routes from $s$ to $t$ considered in $(i)$ and $(ii)$ are composed by simply aggregating entries of $R$, they may not be elementary, i.e., they might visit a same vertex more than once. Then,
$\tilde{\mathcal{K}}$ is not necessarily equal to ${\mathcal{K}}$. Also observe that we only have to compute $\tilde{\mathcal{K}}$ a single time for a given STOP instance, as it is completely based on the original graph $G$.

After that, we build the corresponding 
conflict graph $\tilde{G}_c = (N\backslash\{s,t\},\tilde{\mathcal{K}})$. 
Thereafter, we compute a subset $\tilde{\Sigma} \subseteq \Sigma$ of conflict cliques by finding all the maximal cliques of $\tilde{G}_c $, including the ones of size one. To this end, we apply the depth-first search algorithm of \cite{Tomita2006}, which runs with worst-case time complexity of $\BigO{3^{\frac{|N|}{3}}}.$

\begin{figure}[!ht]
\begin{center}
\scalebox{1}
{
\framebox
{
\begin{minipage}[t]{25cm}
{\small
\begin{tabbing}
xxx\=xxx\=xxx\=xxx\=xxx\=xxx\=xxx\=xxx\=xxx\=xxx\= \kill
\textbf{Input: }A fractional solution $(\bar{x},\bar{y},\bar{f},\bar{\varphi})$, its corresponding residual graph $\tilde{G} = (N,\tilde{A})$ and \\ the subset $\tilde{\Sigma} \subseteq \Sigma$ of maximal conflict cliques. \\
\textbf{Output: }{A set $\mathcal{X}'$ of CCCs violated by $(\bar{x},\bar{y},\bar{f},\bar{\varphi})$.} \\
\textbf{\scriptsize1.} $\mathcal{X}' \leftarrow \emptyset$;\\
\textbf{\scriptsize2.} Set all cliques in $\tilde{\Sigma}$ as \emph{active};\\
\textbf{\scriptsize3.} \textbf{for all }{($\sigma \in \tilde{\Sigma}$)} \textbf{ do}\\
\textbf{\scriptsize4.} \> \> \textbf{if }{($\sigma$ is \emph{active})} \textbf{then}\\
\textbf{\scriptsize5.} \> \> \> \> \textit{Step \rom{1}. Building the auxiliary graphs}\\
\textbf{\scriptsize6.} \> \> \> \>  Build $\tilde{G}'_1 = (\tilde{N}'_1,\tilde{A}'_1)$, $\tilde{N}'_1 = N \cup \{\beta_1\}$, $\tilde{A}'_1 = \tilde{A} \cup \{(i,\beta_1):\, i \in \sigma\}$;\\
\textbf{\scriptsize7.} \> \> \> \>  Build $\tilde{G}'_2 = (\tilde{N}'_2,\tilde{A}'_2)$, $\tilde{N}'_2 = N \cup \{\beta_2\}$, $\tilde{A}'_2 = \{(v,u):\, (u,v) \in \tilde{A}\} \cup \{(i,\beta_2):\, i \in \sigma\}$;\\
\textbf{\scriptsize8.} \> \> \> \> \textbf{for all }{($(u,v) \in \tilde{A}$)} \textbf{ do}\\
\textbf{\scriptsize9.} \> \> \> \> \> \> $c_1'[u,v] \leftarrow c_2'[v,u] \leftarrow c[u,v]$;\\
\textbf{\scriptsize10.} \> \> \> \> \textbf{end-for}; \\
\textbf{\scriptsize11.} \> \> \> \> \textbf{for all }{($i \in \sigma$)} \textbf{ do}\\
\textbf{\scriptsize12.} \> \> \> \> \> \> $c_1'[i,\beta_1] \leftarrow c_2'[i,\beta_2] \leftarrow |M|$;\\
\textbf{\scriptsize13.} \> \> \> \> \textbf{end-for}; \\
\textbf{\scriptsize14.} \> \> \> \> \textit{Step \rom{2}. Looking for a violated CCC (\ref{ccc00})} \\
\textbf{\scriptsize15.} \> \> \> \> $\{ F_{s\to \beta_1},\theta_{s\to \beta_1} \} \leftarrow {\textit{max-flow}}_{s\to \beta_1}{(\tilde{G}_1')}$; \\
\textbf{\scriptsize16.} \> \> \> \> \textbf{if }{($F_{s\to \beta_1} < \sum\limits_{i \in \sigma}{{\bar y}_i}$)} \textbf{ then} \\
\textbf{\scriptsize17.} \> \> \> \> \> \>  $\mathcal{X}' \leftarrow \mathcal{X}' \cup \{ \{ N \backslash \theta_{s\to \beta_1},\sigma \}\}$; \\
\textbf{\scriptsize18.} \> \> \> \> \> \>  Call \emph{update-active-cliques}($\tilde{\Sigma}$, $\sigma$, $\bar{y}$);\\
\textbf{\scriptsize19.} \> \> \> \> \textbf{else} \\
\textbf{\scriptsize20.} \> \> \> \> \> \> \textit{Step \rom{3}. Looking for a violated CCC (\ref{ccc01})} \\
\textbf{\scriptsize21.} \> \> \> \> \> \> $\{ F_{t\to \beta_2},\theta_{t\to \beta_2} \} \leftarrow {\textit{max-flow}}_{t\to \beta_2}{(\tilde{G}_2')}$; \\
\textbf{\scriptsize22.} \> \> \> \> \> \> \textbf{if }{($F_{t\to \beta_2} < \sum\limits_{i \in \sigma}{\bar y}_i$)} \textbf{ then} \\
\textbf{\scriptsize23.} \> \> \> \> \> \> \> \> $\mathcal{X}' \leftarrow \mathcal{X}' \cup \{ \{ N\backslash \theta_{t\to \beta_2},\sigma \}\}$; \\
\textbf{\scriptsize24.} \> \> \> \> \> \> \> \> Call \emph{update-active-cliques}($\tilde{\Sigma}$, $\sigma$, $\bar{y}$);\\
\textbf{\scriptsize25.} \> \> \> \> \> \> \textbf{end-if}; \\
\textbf{\scriptsize26.} \> \> \> \> \textbf{end-if-else} \\
\textbf{\scriptsize27.} \> \> \textbf{end-if};\\
\textbf{\scriptsize28.} \textbf{end-for}; \\
\textbf{\scriptsize29.} \textbf{return} $\mathcal{X}'$;
\end{tabbing}
 }
\end{minipage}
}
}
\end{center}
\caption{Algorithm used to separate violated CCCs.}
\label{ccc_separation}
\end{figure}

Once $\tilde{\Sigma}$ is computed, we look for violated CCCs of types (\ref{ccc00}) and (\ref{ccc01}) as detailed in Figure~\ref{ccc_separation}.
Let the set $\mathcal{X}'$ keep the CCCs found during the separation procedure. Initially, $\mathcal{X}'$ is empty (line 1, Figure~\ref{ccc_separation}). 
Due to the possibly large number of maximal cliques, we adopt a simple filtering mechanism to discard cliques from further search whenever convenient. Precisely, the cliques, which are initially marked as \emph{active} (line 2, Figure~\ref{ccc_separation}), are disabled if any of its vertices belongs to a previously separated CCC.
Then, for every conflict clique $\sigma \in \tilde{\Sigma}$, we check if it is currently \emph{active} (lines 3 and 4, Figure~\ref{ccc_separation}) and, if so,  we build two auxiliary graphs, one for each type of CCC. The first graph, denoted by
$\tilde{G}'_1$, is built by adding to the residual graph $\tilde{G}$ an artificial vertex $\beta_1$ and $|\sigma|$ arcs $(i,\beta_1)$, one for each $i \in \sigma$ (see line 6, Figure~\ref{ccc_separation}). The second one, denoted by $\tilde{G}'_2$, is built by reversing all the arcs of $\tilde{G}$ and, then, adding an artificial vertex $\beta_2$, as well as an arc $(i,\beta_2)$ for each $i \in \sigma$ (see line 7, Figure~\ref{ccc_separation}).

The capacities of the arcs of
$\tilde{G}'_1$ and $\tilde{G}'_2$ are kept in the data structures $c_1'$ and $c_2'$, respectively. In both graphs,
the capacities of the original arcs in $\tilde{G}$ are preserved (see lines 8-10, Figure~\ref{ccc_separation}).
Moreover, all the additional arcs have a same capacity value, which is equal to a sufficiently large number. Here, we adopted the value of $|M|$, the number of vehicles (see lines 11-13, Figure~\ref{ccc_separation}).

Figure~\ref{fig:ex_cccs} illustrates the construction of the auxiliary graphs described above. In this example, we consider the STOP instance of Figure~\ref{fig:ex} and assume that the current fractional solution $(\bar{x},\bar{y},\bar{f},\bar{\varphi})$ has ${\bar x}_{si} = {\bar x}_{ik} = {\bar x}_{kj} = {\bar x}_{sl} = {\bar x}_{lj} = 0.5$, ${\bar x}_{ki} = {\bar x}_{il} = {\bar x}_{lt} = 0$ and ${\bar x}_{jt} = 1$. Moreover, we consider the maximal conflict clique $\{i,k,j\}$ (see, once again, the conflict graph of Figure~\ref{fig:ex:conflictgraph}).

\begin{figure}[!ht]
\begin{subfigure}{\textwidth}
\centering
    \begin{tikzpicture}[
            > = stealth, 
            shorten > = 1pt, 
            auto,
            node distance = 2.5cm, 
            semithick 
        ]

        \tikzstyle{every state}=[
            draw = black,
            thick,
            fill = white,
            minimum size = 4mm
        ]

        \node[state] (s) [fill = gray!20]{$s$};
        \node[state] (i) [above right of=s,fill = gray!20,dashed] {$i$};
        \node[state] (l) [below right of=s] {$l$};
        \node[state] (k) [right of=i,node distance = 2cm,fill = gray!20,dashed] {$k$};
        \node[state] (j) [right of=k,node distance = 2cm,fill = gray!20,dashed] {$j$};
        \node[state] (t) [below right of=j,fill = gray!20] {$t$};
        \node[state] (alpha) [dashed,above of=k,,node distance = 1.5cm]{$\beta_1$};
        
        \node[state] (G) [above of = s, draw = none, node distance = 2.5cm]{$\tilde{G}'_1$};
    
        \path[->] (s) edge node {0.5} (i);
        \path[->] (s) edge node [below left] {0.5} (l);
        \path[->] (l) edge node {0.5} (j);
        \path[->] (i) edge [bend left] node {0.5} (k);
        \path[->] (k) edge node {0.5} (j);
        \path[->] (j) edge node {1} (t);
        
        \path[->] (i) edge[dashed,bend left] node {$\infty$} (alpha);
        \path[->] (j) edge[dashed,bend right] node[above right] {$\infty$} (alpha);
        \path[->] (k) edge[dashed,bend right] node[right] {$\infty$} (alpha);

    \end{tikzpicture}
    \caption{Example of an auxiliary graph $\tilde{G}'_1$ used in the separation of CCCs (\ref{ccc00}).}
  \label{fig:ex_cccs_a}
\end{subfigure}
\begin{subfigure}{\textwidth}
\centering
    \begin{tikzpicture}[
            > = stealth, 
            shorten > = 1pt, 
            auto,
            node distance = 2.5cm, 
            semithick 
        ]

        \tikzstyle{every state}=[
            draw = black,
            thick,
            fill = white,
            minimum size = 4mm
        ]

        \node[state] (s) [fill = gray!20]{$s$};
        \node[state] (i) [above right of=s,fill = gray!20,dashed] {$i$};
        \node[state] (l) [below right of=s] {$l$};
        \node[state] (k) [right of=i,node distance = 2cm,fill = gray!20,dashed] {$k$};
        \node[state] (j) [right of=k,node distance = 2cm,fill = gray!20,dashed] {$j$};
        \node[state] (t) [below right of=j,fill = gray!20] {$t$};
        \node[state] (beta) [dashed,above of=k,,node distance = 1.5cm]{$\beta_2$};
        
        \node[state] (G) [above of = s, draw = none, node distance = 2.5cm]{$\tilde{G}'_2$};
    
        \path[->] (i) edge node [above left] {0.5} (s);
        \path[->] (l) edge node {0.5} (s);
        \path[->] (j) edge node [above left] {0.5} (l);
        \path[->] (k) edge [bend right] node [above] {0.5} (i);
        \path[->] (j) edge node[above] {0.5} (k);
        \path[->] (t) edge node [above right] {1} (j);
        
        \path[->] (i) edge[dashed,bend left] node {$\infty$} (beta);
        \path[->] (j) edge[dashed,bend right] node[above right] {$\infty$} (beta);
        \path[->] (k) edge[dashed,bend right] node[right] {$\infty$} (beta);

    \end{tikzpicture}
  \caption{Example of an auxiliary graph $\tilde{G}'_2$ used in the separation of CCCs (\ref{ccc01}).}
  \label{fig:ex_cccs_b}
  \end{subfigure}
  
  \caption{Auxiliary graphs built when considering the STOP instance of Figure~\ref{fig:ex}, the maximal conflict clique $\{i,k,j\}$ and a fractional solution $\bar{x}$, with ${\bar x}_{si} = {\bar x}_{ik} = {\bar x}_{kj} = {\bar x}_{sl} = {\bar x}_{lj} = 0.5$, ${\bar x}_{ki} = {\bar x}_{il} = {\bar x}_{lt} = 0$ and ${\bar x}_{jt} = 1$. Here, the values associated with the arcs are their corresponding capacities, and the infinity symbol stands for a sufficiently large value.}
  \label{fig:ex_cccs}
\end{figure}
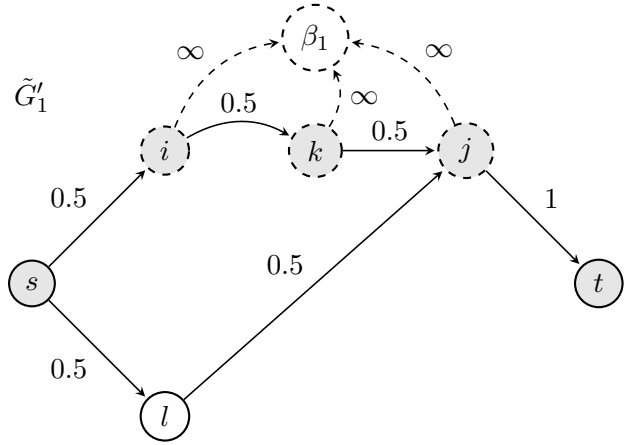
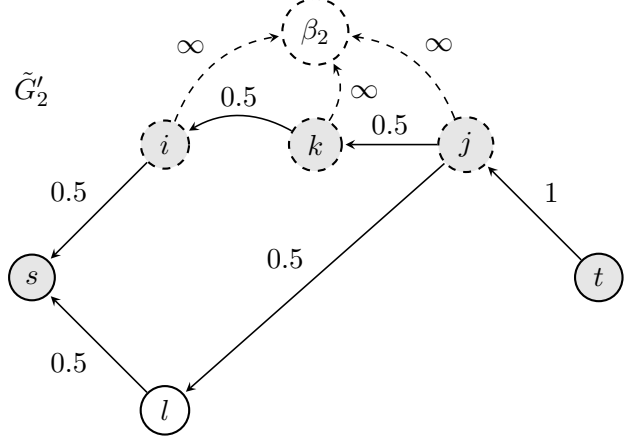

Once the auxiliary graphs are built for a given $\sigma \in \tilde{\Sigma}$, the algorithm first looks for a violated CCC (\ref{ccc00}) by computing the maximum flow from $s$ to $\beta_1$ on $\tilde{G}'_1$. Accordingly, let $\{ F_{s\to \beta_1},\theta_{s\to \beta_1} \}$ be the solution of ${\textit{max-flow}}_{s\to \beta_1}{(\tilde{G}'_1)}$.
Recall that $F_{s\to \beta_1}$ gives the value of the resulting maximum flow, and $\theta_{s\to \beta_1}$ defines a corresponding minimum cut, with $s \in \theta_{s \to \beta_1}$. Then, the algorithm checks if $F_{s\to \beta_1}$ is smaller than $\sum\limits_{i \in \sigma}{{\bar y}_i}$. If that is the case, a violated CCC
(\ref{ccc00}) is identified and added to $\mathcal{X}'$. Precisely, this inequality is denoted by
$\{ N\backslash \theta_{s\to \beta_1},\sigma \}$ and defined as
\begin{equation}
\sum\limits_{e \in \delta^-(N \backslash \theta_{s \to \beta_1})}{x_{e}} \geq \sum\limits_{i \in \sigma}y_{i},
\end{equation}

\noindent where $N \backslash \theta_{s \to \beta_1}$ corresponds to the subset $V \subseteq N\backslash \{s\}$ of (\ref{ccc00}). If a violated CCC (\ref{ccc00}) is identified, we also disable the cliques that are no longer \emph{active} by calling the procedure \emph{update-active-cliques}, which is described in Figure~\ref{ccc_update-active-cliques}. The separation of CCCs (\ref{ccc00}) is summarized at lines 14-19, Figure~\ref{ccc_separation}.
\begin{figure}[!ht]
\begin{center}
\scalebox{1}
{
\framebox
{
\begin{minipage}[t]{25cm}
{\small
\begin{tabbing}
xxx\=xxx\=xxx\=xxx\=xxx\=xxx\=xxx\=xxx\=xxx\=xxx\= \kill
\textbf{Input: }A set of conflict cliques $\tilde{\Sigma}$, a conflict clique $\sigma \in \tilde{\Sigma}$ and a fractional solution $\bar{y}$. \\
\textbf{\scriptsize1.} \textbf{for all} (vertex {$i \in \sigma$}) \textbf{do}\\
\textbf{\scriptsize2.} \> \> \textbf{if }{($\bar{y}_i > 0$)} \textbf{then}\\
\textbf{\scriptsize3.} \> \> \> \> Deactivate every clique in $\tilde{\Sigma}$ containing $i$;\\
\textbf{\scriptsize4.} \> \> \textbf{end-if};\\
\textbf{\scriptsize5.} \textbf{end-for};
\end{tabbing}
 }
\end{minipage}
}
}
\end{center}
\caption{Procedure \emph{update-active-cliques}, which manages the currently \emph{active} conflict cliques during the separation of CCCs.}
\label{ccc_update-active-cliques}
\end{figure}

Notice that, if the separation algorithm does not find a violated CCC (\ref{ccc00}) for the current clique, then this clique remains \emph{active}. In this case, the algorithm looks for a violated CCC of type (\ref{ccc01}) by computing $\{ F_{t\to \beta_2},\theta_{t\to \beta_2} \} = {\textit{max-flow}}_{t\to \beta_2}{(\tilde{G}'_2)}$. If $F_{t\to \beta_2}$ is smaller than $\sum\limits_{i \in \sigma}{{\bar y}_i}$, then a violated CCC (\ref{ccc01}) is identified and added to $\mathcal{X}'$. This CCC is denoted by $\{ N\backslash \theta_{t\to \beta_2},\sigma\}$ and defined as
\begin{equation}
\sum\limits_{e \in \delta^+(N \backslash \theta_{t \to \beta_2})}{x_{e}} \geq \sum\limits_{i \in \sigma}y_{i}.
\end{equation}

\noindent Here, $N \backslash \theta_{t \to \beta_2}$ corresponds to the subset $V \subseteq N\backslash \{t\}$ of (\ref{ccc01}). Also in this case, we disable the cliques that are no longer \emph{active} by calling the procedure \emph{update-active-cliques}.
The separation of CCCs (\ref{ccc01}) is summarized at lines 20-25, Figure~\ref{ccc_separation}.

\subsubsection{Arc-Vertex Inference Cuts (AVICs)}
Consider the set $E = \{\{i,j\} :\, (i,j) \in {A}\mbox{ and } (j,i) \in {A}\}$. AVICs are defined as 
\begin{align}
\lvert y_i - y_j \rvert \leq 1 - (x_{ij} + x_{ji}) & \qquad \forall \, \{i,j\} \in E, \label{avics00}
\end{align}
\noindent which can be linearized as
\begin{align}
y_i - y_j \leq 1 - (x_{ij}+x_{ji}) \qquad \forall \, \{i,j\} \in E,& \mbox{ and} \label{avics01} \\
y_j - y_i \leq 1 - (x_{ij}+x_{ji}) \qquad \forall \, \{i,j\} \in E.& \label{avics02}
\end{align}

\noindent In logic terms, these inequalities correspond to the boolean expressions
\begin{align*}
(x_{ij} = 1 \oplus x_{ji} = 1) \to y_i - y_j = 0 & \qquad \forall \, \{i,j\} \in E,
\end{align*}
\noindent where $\oplus$ stands for the \emph{exclusive disjunction} operator.

The validity of AVICs relies on two trivial properties that are inherent to feasible solutions for $\mathcal{F}$: $(i)$ for all $\{i,j\} \in E$, arcs $(i,j)$ and $(j,i)$ cannot be simultaneously selected (i.e., $x_{ij} + x_{ji} \leq 1$), and $(ii)$ once an arc $(i,j) \in A$ is selected in a solution (i.e., $x_{ij} = 1$), we must have $y_i = y_j$ (which, more precisely, are also equal to one).

One may notice that simpler valid inequalities can be devised by considering arcs separately, as follows
\begin{align}
\lvert y_i - y_j \rvert \leq 1 - x_{ij} \qquad \forall \, (i,j) \in A. \label{avics03}
\end{align}
\noindent However, these inequalities are not only weaker, but redundant for formulation $\mathcal{F}$, as proven next.

\begin{proposition}
Inequalities (\ref{avics03}) do not cut off any solution from the polyhedron
$\mathcal{L}$.
\end{proposition}

\begin{proof}
From (\ref{b102}), (\ref{b105}) and the domain of the binary variables $x$ and $y$, we have that $0 \leq x_{ij} \leq y_i \leq 1$ and $0 \leq x_{ij} \leq y_j \leq 1$ for all $(i,j) \in A$. Accordingly, the maximum possible value assumed by $\lvert y_i - y_j \rvert$ occurs when one of the corresponding $y$ variables assumes its minimum value (i.e., $x_{ij}$), and the other assumes its maximum (i.e., $1$).
\end{proof}

\subsubsection{Separation of AVICs}
\label{s_sep_avics}
Notice that the total number of AVICs (\ref{avics01}) and (\ref{avics02}) is equal to $2\times|E|$, which is at most $|A|$. Then, this family of inequalities can be separated by complete enumeration with \BigO{|A|} time complexity.

\subsubsection{Lifted Cover Inequalities (LCIs)}
First, consider the knapsack inequality
\begin{equation}
\sum \limits_{i \in P}{p_{i}y_{i}} \leq \lfloor\tau\rfloor \label{lci1},
\end{equation}

\noindent where $\tau$ is a dual (upper) bound on the optimal solution value of $\mathcal{F}$, and
$P \subseteq N \backslash \{s,t\}$ is the set of profitable vertices, with $p_i \in \mathbb{Z}^+$ for all $i \in P$, as defined in Section~\ref{s_notation}. By definition, (\ref{lci1}) is valid for $\mathcal{F}$, once its left-hand side corresponds to the objective function of this formulation. 

Based on (\ref{lci1}), we devise LCIs, which are classical cover inequalities strengthened through lifting (see, e.g., \cite{Balas1975,Wolsey1975,Gu98,Kaparis08}). Precisely, we depart from cover inequalities of the form
\begin{equation}
\sum \limits_{i \in C}{y_{i}} \leq |C| -1 \label{lci2},
\end{equation}

\noindent where the set $C \subseteq P$ \emph{covers} (\ref{lci1}), i.e., $\sum \limits_{i \in C}{p_{i}} > \lfloor\tau\rfloor$. Moreover, $C$ must be minimal, i.e., $\sum \limits_{i \in C\backslash\{j\}}{p_{i}} \leq \lfloor\tau\rfloor$ for all $j \in C$. Then, let the disjoint sets $C_1$ and $C_2$ define a partition of $C$, with $C_1 \neq \emptyset$. LCIs are defined as
\begin{equation}
\sum \limits_{i \in C_1}{y_{i}} + \sum\limits_{j \in C_2}{\pi_jy_j} + \sum\limits_{j \in P\backslash C}{\mu_jy_j} \leq |C_1| + \sum\limits_{j \in C_2}{\pi_j} -1, \label{lci3}
\end{equation}

\noindent where $\pi_i \in \mathbb{Z}$, $\pi_i \geq 1$ for all $i\in C_2$, and $\mu_i \in \mathbb{Z}$, $\mu_i \geq 0$ for all $ i\in P \backslash C$, are the lifted coefficients.
One may note that setting $\pi = \textbf{1}$ and $\mu = \textbf{0}$ reduces ($\ref{lci3}$) to the classical cover inequality (\ref{lci2}).

\subsubsection{Separation of LCIs}
\label{s_sep_lcis}

The separation algorithm we adopt follows the classical algorithmic framework of \cite{Gu98}.
In general terms, the liftings are done sequentially (i.e., one variable at a time) according to the values of the $y$ variables in the current fractional solution $(\bar{x},\bar{y},\bar{f},\bar{\varphi})$, and the lifted coefficients are computed by solving auxiliary knapsack problems to optimality. Once all the variables are tested for lifting, the resulting LCI is checked for violation. We refer to \cite{Assuncao2019} for a detailed explanation of the separation algorithm and a didactic overview on the concepts of lifting.

\section{Reinforcing the original formulation through cutting planes}
\label{s_cutting-plane}
In this section, we briefly describe the cutting-plane algorithm used to reinforce formulation $\mathcal{F}$ through the addition of the valid inequalities discussed in Section~\ref{s_cuts}. The algorithm follows the same framework adopted by \cite{Assuncao2019}, except for the types of cuts separated. For simplicity, we assume that $\mathcal{F}$ is feasible.

The algorithm starts by adding to $\mathcal{L}$ --- the linear relaxation of $\mathcal{F}$ --- all AVICs, which are separated by complete enumeration (as discussed in Section~\ref{s_sep_avics}). 
Then, an iterative procedure takes place. Precisely, at each iteration, the algorithm looks for CCCs and LCIs violated by the solution of the current LP model. These cuts are found by means of the separation procedures described in Sections~\ref{s_sep_cccs} and~\ref{s_sep_lcis}. Instead of selecting all the violated CCCs found, we only add to the model the most violated cut (if any) and the ones that are sufficiently orthogonal to it. Such strategy is able to balance the strength and diversity of the cuts separated, while limiting the model size (see, e.g., \cite{Wesselmann2012,Samer2015,Bicalho2016}). Naturally, this filtering procedure does not apply to LCIs, since at most a single LCI is separated per iteration.
Details on how these cuts are selected are given in Section~\ref{s_implementation_details}.

The algorithm iterates until either no more violated cuts are found or the bound improvement of the current model --- with respect to the model from the previous iteration --- is inferior or equal to a tolerance $\epsilon_1$. The order in which CCCs and LCIs are separated is not relevant, since the bound provided by the current LP model is only updated at the end of each loop, when all separation procedures are done.

\section{Feasibility Pump (FP)}
\label{s_FP}
In this section, we describe the FP algorithm we adopt to find initial solutions for STOP by means of formulation $\mathcal{F}$ and its reinforced version obtained from the cutting-plane algorithm previously presented.
The FP matheuristic was proposed by \cite{Fischetti2005} as an alternative to solve the NP-hard problem of finding feasible solutions for generic MILP problems of the form $\min\{c^Tx:\, Ax \geq b,\, x_i \mbox{ integer } \forall i \in \mathcal{I}\}$. Here, $x$ and $c$ are column vectors of, respectively, variables and their corresponding costs, $A$ is the restriction matrix, $b$ is a column vector, and $\mathcal{I}$ is the set of integer variables.

At each iteration, also called \emph{pumping cycle} (or, simply, \emph{pump}) of FP, an integer (infeasible) solution $\tilde{x}$ is used to build an auxiliary LP problem based on the linear relaxation of the original MILP problem. Precisely, the auxiliary problem aims at finding a solution ${x}^*$ with minimum distance from $\tilde{x}$ in the search space defined by $\{x:\, Ax \geq b\}$. Each new ${x}^*$ is rounded and used as the integer solution of the next iteration. The algorithm ideally stops when the current solution of the auxiliary problem is also integer (i.e.,  $\left[x^*\right]= x^{*}$, where $\left[x^*\right] = \lfloor x^*\,+\,0.5\rfloor$) and, thus, feasible for the original problem. Notice that FP only works in the continuous space of solutions that satisfy all the linear constraints of the original problem, and the objective function of the auxiliary problems is the one element that guides the fractional solutions into integer feasibility. 

The original FP framework pays little attention to the quality of the solutions. In fact, the objective function of the original MILP problem is only taken into account to generate an initial fractional solution to be rounded and used in the first iteration. In all the subsequent iterations, the auxiliary LPs aim at minimizing distance functions that do not explore the original objective, which explains the poor quality of the solutions obtained \citep{Fischetti2005,Achterberg2007}. Some variations of the framework address this issue by combining, in the auxiliary problems, the original objective function with the distance metric. That is the case, for instance, of the Objective Feasibility Pump (OFP) \citep{Achterberg2007}, in which the transition from the original objective function to the distance-based auxiliary one is done gradually with the progress of the pumps.
We refer to \cite{Berthold2019} for a detailed survey on the several FP variations that have been proposed throughout the years to address possible drawbacks and convergence issues of the original framework.

In this work, we adopt both the original FP framework and OFP to find feasible solutions for $\mathcal{F}$. In the sequel, we only describe in details OFP, as it naturally generalizes FP. For simplicity, formulation $\mathcal{F}$ is used throughout the explanation, instead of a generic MILP.

Consider the vector $x$ of decision variables (one for each arc in $A$), as defined in Section~\ref{s_models}, and let $\tilde{x} \in \{0,1\}^{|A|}$ be a binary vector defining a not necessarily feasible solution for $\mathcal{F}$. The distance function used to guide the OFP framework into integer feasibility is defined as
\begin{eqnarray}
\Delta(x,\tilde{x}) = \sum\limits_{(i,j) \in A}\mid x_{ij}- \tilde{x}_{ij}\mid,
\end{eqnarray}
\noindent which can be rewritten in a linear manner as
\begin{eqnarray}
\Delta(x,\tilde{x}) = \sum\limits_{(i,j) \in A:\,\tilde{x}_{ij} = 0} {x_{ij}} + \sum\limits_{(i,j) \in A:\,\tilde{x}_{ij} = 1} {(1 - x_{ij})}.
\end{eqnarray}

Considering the $y$ decision variables on the selection of vertices in the solution routes (as defined in Section~\ref{s_models}), the objective function of the auxiliary problems solved at each iteration of OFP consists of a convex combination of the distance function $\Delta(x,\tilde{x})$ and the original objective function of $\mathcal{F}$. Precisely, 
\begin{eqnarray}
\Delta_{\gamma}(x,y,\tilde{x}) = \frac{(1- \gamma)}{\lVert \Delta(x,\tilde{x}) \rVert}\Delta(x,\tilde{x}) + \frac{\gamma}{\lVert \sum\limits_{i \in P}p_iy_i \rVert}\overbrace{(-\sum\limits_{i \in P}p_iy_i)}^{\mbox{\tiny minus } (\ref{csc100})} \label{ofp00},
\end{eqnarray}
\noindent with $\gamma \in [0,1]$, and $\lVert \cdot \rVert$ being the Euclidean norm of a vector. Notice that, in (\ref{ofp00}), we consider an alternative definition of $\mathcal{F}$ as a minimization problem, in which $\max \sum\limits_{i \in P}p_iy_i = \min (-\sum\limits_{i \in P}p_iy_i)$. Moreover, both $\Delta(x,\tilde{x})$ and the original objective function are normalized in order to avoid scaling issues. Also notice that, since $\tilde{x}$ is a constant vector, $\lVert \Delta(x,\tilde{x}) \rVert = \sqrt{|A|}$ in this case.

Now, consider the polyhedron $\Omega$ defined by the feasible region of $\mathcal{L}$, the linear relaxation of $\mathcal{F}$. Precisely, $\Omega = \{(x,y,f,\varphi) \in \mathbb{R}^{|A|}\times\mathbb{R}^{|N|}\times\mathbb{R}^{|A|}\times\mathbb{R}:\, (\ref{b101})$-$(\ref{csc110}),\,(\ref{csc114}),\, \textbf{0}\leq x \leq \textbf{1},\, \textbf{0}\leq y \leq \textbf{1} \mbox{ and } f \geq \textbf{0}\}$.
OFP works by iteratively solving auxiliary problems defined as
\begin{eqnarray}
D(x,y,\tilde{x},\gamma):\,\min\{\Delta_{\gamma}(x,y,\tilde{x}):\, ({x},{y},{f},{\varphi}) \in \Omega \},
\end{eqnarray}

\noindent where $\gamma$ balances the influence of the distance function and the original objective, i.e., the integer feasibility and the quality of the solution. Considering $\mathcal{F}$, the OFP algorithm is described in Figure~\ref{fig_ofp}.

\begin{figure}[!ht]
\begin{center}
\scalebox{1}
{
\framebox
{
\begin{minipage}[t]{25cm}
{\small
\begin{tabbing}
xxx\=xxx\=xxx\=xxx\=xxx\=xxx\=xxx\=xxx\=xxx\=xxx\= \kill
\textbf{Input: }{The model $\mathcal{F}$, \emph{max\_pumps} $\in \mathbb{Z}^+$, \emph{max\_pumps} $\geq 1$, $\lambda \in [0,1]$ and $K \in \mathbb{Z}^+$.} \\
\textbf{Output: }{Ideally, a feasible solution for $\mathcal{F}$.} \\
\textbf{\scriptsize1.} Initialize $\gamma \leftarrow 1$ and \emph{iter\_counter} $\leftarrow$ 1;\\
\textbf{\scriptsize2.} Solve $D(x,y,\textbf{0},\gamma)$, obtaining a solution $({x}^*,{y}^*,{f}^*,{\varphi}^*)$;\\
\textbf{\scriptsize3.} \textbf{if }{${x}^*$ is integer} \textbf{ then} \textbf{return} ${x}^*$;\\
\textbf{\scriptsize4.} $\tilde{x} \leftarrow \left[x^*\right]$ (= rounding of $x^*$);\\
\textbf{\scriptsize5.} \textbf{while} {(\emph{iter\_counter} $\leq$ \emph{max\_pumps})};\\
\textbf{\scriptsize6.} \> \> Update $\gamma \leftarrow \lambda\gamma$ and \emph{iter\_counter} $\leftarrow$ \emph{iter\_counter} + 1;\\
\textbf{\scriptsize7.} \> \> Solve $D(x,y,\tilde{x},\gamma)$, obtaining a solution $({x}^*,{y}^*,{f}^*,{\varphi}^*)$;\\
\textbf{\scriptsize8.} \> \> \textbf{if }{(${x}^*$ is integer)} \textbf{then} \textbf{return} ${x}^*$;\\
\textbf{\scriptsize9.} \> \> \textbf{if }{($\tilde{x} \neq \left[x^*\right]$)} \textbf{ then} {$\tilde{x} \leftarrow \left[x^*\right]$};\\
\textbf{\scriptsize10.} \> \> \textbf{else }{flip \emph{rand}$(K/2,\,3K/2)$ entries $\tilde{x}_{ij}$, $(i,j) \in A$, with highest $\mid x^*_{ij} - \tilde{x}_{ij} \mid$};\\
\textbf{\scriptsize11.} \textbf{end-while};\\
\textbf{\scriptsize12.} \textbf{return 0};
\end{tabbing}
 }
\end{minipage}
}
}
\end{center}
\caption{Description of the OFP algorithm when considering formulation $\mathcal{F}$.}
\label{fig_ofp}
\end{figure}

Aside from the corresponding model $\mathcal{F}$, the algorithm receives as input three values: the maximum number of iterations (pumps) to be performed (\emph{max\_pumps}), a rate by which the $\gamma$ value is decreased at each pump ($\lambda$) and a basis value ($K$) used to compute the amplitude of the perturbations to be performed in solutions that cycle.

At the beginning, $\gamma$ and a variable that keeps the number of the current iteration (\emph{iter\_counter}) are both set to one (line 1, Figure~\ref{fig_ofp}). Then, the current problem $D(x,y,\textbf{0},\gamma)$ is solved, obtaining a solution $({x}^*,{y}^*,{f}^*,{\varphi}^*)$. Notice that, since $\gamma =1$ at this point, $D(x,y,\textbf{0},\gamma)$ corresponds to $\mathcal{L}$, and the integer solution $\textbf{0}$ plays no role. If $x^*$ is integer, and, thus, the current solution is feasible for $\mathcal{F}$, the algorithm stops. Otherwise, the rounded value of $x^*$ is kept in a vector $\tilde{x}$ (see lines 2-4, Figure~\ref{fig_ofp}).

After the first pump, an iterative procedure takes place until either an integer feasible solution is found or the maximum number of iterations is reached. At each iteration, the $\gamma$ value is decreased by the fixed rate $\lambda$, and the iteration counter is updated. Then, $D(x,y,\tilde{x},\gamma)$ is solved, obtaining a solution $({x}^*,{y}^*,{f}^*,{\varphi}^*)$. If $x^*$ is integer at this point, the algorithm stops. Otherwise, it checks if the algorithm is caught up in a cycle of size one, i.e., if $\tilde{x}$ (the rounded solution from the previous iteration) is equal to $\left[x^*\right]$. If not, $\tilde{x}$ is simply updated to $\left[x^*\right]$. In turn, if a cycle is detected, the algorithm performs a perturbation on $\tilde{x}$. Precisely, a random integer in the open interval $(K/2,\,3K/2)$ is selected as the quantity of binary entries $\tilde{x}_{ij}$, $(i,j) \in A$, to be flipped to the opposite bound. This perturbation prioritizes entries that have highest values in the distance vector $\mid x^* - \tilde{x}\mid$. The loop described above is summarized at lines 5-11, Figure~\ref{fig_ofp}. At last, if no feasible solution is found within \emph{max\_pumps} iterations, the algorithm terminates with a null solution (line 12, Figure~\ref{fig_ofp}).

The original FP framework follows the same algorithm described in Figure~\ref{fig_ofp}, with the exception that the decrease rate $\lambda$ given as input is necessarily zero. Then, in the loop of lines 5-11, the $D(x,y,\tilde{x},\gamma)$ problems are solved under $\gamma = 0$, i.e., without taking into account the original objective function.

\section{A Large Neighborhood Search (LNS) heuristic with Path Relinking (PR)}
\label{s_LNS}
In this section, we describe an LNS heuristic for STOP, which we apply to improve the quality of the initial solutions obtained from the OFP algorithm described in the previous section. The original LNS metaheuristic framework \citep{Shaw1998} works by gradually improving an initial solution through a sequence of destroying and repairing procedures. In our heuristic, the LNS framework is coupled with classical local search procedures widely used to improve solutions of routing problems in general. In particular, these procedures, which are described in Section~\ref{s_local_searches}, are also present in most of the successful heuristics proposed to solve TOP (e.g., \cite{Vansteenwegen09,Ke08,Souffriau10,Kim13,Dang13,Ke2016}). The heuristic we propose also uses a memory component known as Path Relinking (PR). PR was devised by \cite{Glover1997} and its original version explores a neighborhood defined by the set of intermediate solutions --- namely, the ``path'' --- between two given solutions. The PR framework has also been successfully applied to solve TOP \citep{Souffriau10}. 

\subsection{Main algorithm}
We describe, in Figure~\ref{fig_LNS}, the general algorithm of the LNS heuristic we propose. The heuristic receives four inputs: an initial feasible solution --- built through OFP of FP---, the number of iterations to be performed (\emph{max\_iter}), the capacity of the \emph{pool} of solutions (\emph{max\_pool\_size}) and a parameter called \emph{stalling\_limit}, which manages how frequently the PR procedure is called. Precisely, it limits the number of iterations in stalling (i.e., with no improvement in the current best solution) before calling the PR procedure.
Initially, the variables that keep the current number of iterations (\emph{iter\_counter}) and the number of iterations since the last solution improvement (\emph{stalling\_counter}) are set to zero (line 1, Figure~\ref{fig_LNS}). Then, the initial solution $Y$ is improved through local search procedures (detailed in Section~\ref{s_local_searches}) and added to the initially empty \emph{pool} of solutions $\Lambda$ (lines 2 and 3, Figure~\ref{fig_LNS}).
\begin{figure}[!ht]
\begin{center}
\scalebox{1}
{
\framebox
{
\begin{minipage}[t]{25cm}
{\small
\begin{tabbing}
xxx\=xxx\=xxx\=xxx\=xxx\=xxx\=xxx\=xxx\=xxx\=xxx\= \kill
\textbf{Input: }{An initial feasible solution $Y$, \emph{max\_iter} $\in \mathbb{Z}^+$, \emph{max\_iter} $\geq 1$, \emph{max\_pool\_size} $\in \mathbb{Z}^+$},\\ {\emph{max\_pool\_size} $\geq 1$ and \emph{stalling\_limit} $\in \mathbb{Z}^+$, \emph{stalling\_limit} $\geq 1$.} \\
\textbf{Output: }{An ideally improved feasible solution.} \\
\textbf{\scriptsize1.} Initialize \emph{iter\_counter} $\leftarrow 0$ and \emph{stalling\_counter} $\leftarrow 0$;\\
\textbf{\scriptsize2.} Improve $Y$ through local searches (see Section~\ref{s_local_searches});\\
\textbf{\scriptsize3.} Initialize the \emph{pool} of solutions $\Lambda \leftarrow \{Y\}$;\\
\textbf{\scriptsize4.} \textbf{while} {(\emph{iter\_counter} $\leq$ \emph{max\_iter})};\\
\textbf{\scriptsize5.} \> \> Update $\emph{iter\_counter} \leftarrow \emph{iter\_counter} + 1$;\\
\textbf{\scriptsize6.} \> \> Randomly select a solution $Y'$ from $\Lambda$;\\
\textbf{\scriptsize7.} \> \> Partially destroy $Y'$ by removing vertices (see Section~\ref{s_destroy});\\
\textbf{\scriptsize8.} \> \> \textbf{do}\\
\textbf{\scriptsize9.} \> \> \> \> Improve $Y'$ through local searches (Section~\ref{s_local_searches});\\
\textbf{\scriptsize10.} \> \> \> \> \textbf{if} {($Y'$ is better than the best solution in $\Lambda$)} \textbf{then}\\
\textbf{\scriptsize11.} \> \> \> \> \> \> \emph{stalling\_counter} $\leftarrow -1$;\\
\textbf{\scriptsize12.} \> \> \> \> \> \> \textbf{if} {($|\Lambda| < \emph{max\_pool\_size}$)} \textbf{then} {$\Lambda \leftarrow \Lambda \cup \{Y'\}$};\\
\textbf{\scriptsize13.} \> \> \> \> \> \> \textbf{else} {Replace the worst solution in $\Lambda$ with $Y'$};\\
\textbf{\scriptsize14.} \> \> \> \> \textbf{end-if};\\
\textbf{\scriptsize15.} \> \> \textbf{while} {(perform inter-route \emph{shifting} perturbations on $Y'$) (Section~\ref{s_shifting})};\\
\textbf{\scriptsize16.} \> \> \textbf{if} {($Y'$ is better than the best solution in $\Lambda$)} \textbf{then} \\
\textbf{\scriptsize17.} \> \> \> \> $\emph{stalling\_counter} \leftarrow 0$;\\
\textbf{\scriptsize18.} \> \> \textbf{else} $\emph{stalling\_counter} \leftarrow \emph{stalling\_counter} + 1$;\\
\textbf{\scriptsize19.} \> \> \textbf{if} {(\emph{stalling\_counter} $\geq$ \emph{stalling\_limit})} \textbf{then} \\
\textbf{\scriptsize20.} \> \> \> \> {Perform the PR procedure considering $\Lambda$ and $Y'$ (see Section~\ref{s_pr})};\\
\textbf{\scriptsize21.} \> \> \> \> $\emph{stalling\_counter} \leftarrow 0$;\\
\textbf{\scriptsize22.} \> \> \textbf{end-if};\\
\textbf{\scriptsize23.} \> \> \textbf{if} {($Y'\notin \Lambda$)} \textbf{and} {($Y'$ is better than the worst solution in $\Lambda$)} \textbf{then} \\
\textbf{\scriptsize24.} \> \> \> \> \textbf{if} {($|\Lambda| < \emph{max\_pool\_size}$)} \textbf{then} {$\Lambda \leftarrow \Lambda \cup \{Y'\}$};\\
\textbf{\scriptsize25.} \> \> \> \> \textbf{else} {Replace the worst solution in $\Lambda$ with $Y'$};\\
\textbf{\scriptsize26.} \> \> \textbf{end-if};\\
\textbf{\scriptsize27.} \textbf{end-while};\\
\textbf{\scriptsize28.} \textbf{return} {best solution in $\Lambda$};
\end{tabbing}
 }
\end{minipage}
}
}
\end{center}
\caption{Description of the general LNS algorithm.}
\label{fig_LNS}
\end{figure}

At this point, an iterative procedure is performed \emph{max\_iter} times (lines 4-27, Figure~\ref{fig_LNS}).
First, the iteration counter is incremented, and a solution $Y'$ (randomly selected from the \emph{pool}) is partially destroyed by the removal of some vertices, as later described in Section~\ref{s_destroy}. 
Then, the algorithm successively tries to improve $Y'$ (lines 8-15, Figure~\ref{fig_LNS}). To this end, the local searches of Section~\ref{s_local_searches} are performed on $Y'$. If the improved $Y'$ is better than the best solution currently in the pool (i.e., its total profit sum is strictly greater), then the stalling counter is set to -1, and a copy of $Y'$ is added to $\Lambda$. The addition of solutions to the pool always considers its capacity. Accordingly, if the pool is not full, the new solution is simply added. Otherwise, it takes the place of the current worst solution in the pool (see lines 10-14, Figure~\ref{fig_LNS}).

At this point, the algorithm attempts to do vertex shifting perturbations (line 15, Figure~\ref{fig_LNS}), which are detailed in Section~\ref{s_shifting}. If it succeeds,
the algorithm resumes to another round of local searches.
Otherwise, the main loop proceeds by updating the stalling counter. Precisely, if the possibly improved $Y'$ obtained after the successive local searches and shifting perturbations has greater profit sum than the best solution in $\Lambda$, then \emph{stalling\_counter} is reset to zero. Otherwise, it is incremented by one (lines 16-18, Figure~\ref{fig_LNS}).
After that, the algorithm checks if the limit number of iterations in stalling was reached. If that is the case, the PR procedure, whose description is given in Section~\ref{s_pr}, is applied in the current iteration, and the stalling counter is reset once again (lines 19-22, Figure~\ref{fig_LNS}).

By the end of the main loop, it is checked if $Y'$ should be added to $\Lambda$, which only occurs if the solution does not already belong to the pool and its profit sum is greater than the current worst solution available (lines 23-26, Figure~\ref{fig_LNS}).
At last, the algorithm returns the best solution in the pool (line 28, Figure~\ref{fig_LNS}). In the next sections, we detail all the aforementioned procedures called within the heuristic.

\subsection{Destroying procedure}
\label{s_destroy}
The destroying procedure consists of removing some of the profitable vertices belonging to a given feasible solution. Consider a fixed parameter $\emph{removal\_percentage} \in [0,1]$. First, we determine an upper bound on the number of vertices to be removed, namely \emph{max\_number\_of\_removals}. This value is randomly selected in the open interval defined by zero and the product of \emph{removal\_percentage} and the quantity of profitable vertices in the solution. Then, the procedure sequentially performs \emph{max\_number\_of\_removals} attempts of vertex removal, such that, at each time, a visited vertex is randomly selected. Nevertheless, a vertex is only actually removed if it is profitable and the resulting route remains feasible.

\subsection{Insertion procedure}
\label{s_insertions}
Every time the insertion procedure is called, one of two possible priority orders on the unvisited vertices to the inserted is randomly chosen: non-increasing or non-decreasing orders of profits. Then, according to the selected order, the unvisited vertices are individually tested for insertion in the current solution. If a given vertex can be added to the solution (i.e., its addition does not make the routes infeasible), it is inserted in the route and position that increase the least the sum of the routes' time durations. Otherwise, the vertex remains unvisited and the next one in the sequence is tested for insertion.

\subsection{Local searches}
\label{s_local_searches}

Given a feasible solution, the local searches here adopted attempt to improve the solution quality, which, in this case, means either increasing the profit sum or decreasing the sum of the routes' times (while maintaining the same profit sum). The general algorithm sequentially performs inter and intra-route improvements, vertex replacements and attempts of vertex insertions, as summarized in Figure~\ref{fig_local_searches}.
The inter and intra-route improvements are detailed in Section~\ref{s_inter_intra_route_improvements}, while the vertex replacements are described in Section~\ref{s_vertex_replacements}. The attempts of vertex insertions (lines 3 and 5, Figure~\ref{fig_local_searches}) are done as described in Section~\ref{s_insertions}. The algorithm stops when it reaches a locally optimal solution with respect to the neighborhoods defined by the aforementioned improvement procedures, i.e., no more improvements are achieved.

\begin{figure}[!ht]
\begin{center}
\scalebox{1}
{
\framebox
{
\begin{minipage}[t]{25cm}
{\small
\begin{tabbing}
xxx\=xxx\=xxx\=xxx\=xxx\=xxx\=xxx\=xxx\=xxx\=xxx\= \kill
\textbf{Input: }{An initial feasible solution $Y$.} \\
\textbf{Output: }{A possibly improved version of $Y$.} \\
\textbf{\scriptsize1.} \textbf{do}\\
\textbf{\scriptsize2.} \> \> Do inter and intra-route improvements (Section~\ref{s_inter_intra_route_improvements});\\
\textbf{\scriptsize3.} \> \> Try to insert unvisited vertices in $Y$ (Section~\ref{s_insertions});\\
\textbf{\scriptsize4.} \> \> Do vertex replacements (Section~\ref{s_vertex_replacements});\\
\textbf{\scriptsize5.} \> \> Try to insert unvisited vertices in $Y$ (Section~\ref{s_insertions});\\
\textbf{\scriptsize6.} \textbf{while} (did any improvement on $Y$); \\
\textbf{\scriptsize7.} \textbf{return} {$Y$};
\end{tabbing}
 }
\end{minipage}
}
}
\end{center}
\caption{Description of the sequence of local searches.}
\label{fig_local_searches}
\end{figure}

\subsubsection{Inter and intra-route improvements}
\label{s_inter_intra_route_improvements}
\begin{figure}[!ht]
\begin{center}
\scalebox{1}
{
\framebox
{
\begin{minipage}[t]{25cm}
{\small
\begin{tabbing}
xxx\=xxx\=xxx\=xxx\=xxx\=xxx\=xxx\=xxx\=xxx\=xxx\= \kill
\textbf{Input: }{An initial feasible solution $Y$.} \\
\textbf{Output: }{A possibly improved version of $Y$.} \\
\textbf{\scriptsize1.} \textbf{do}\\
\textbf{\scriptsize2.} \> \> \emph{// inter-route improvements}\\
\textbf{\scriptsize3.} \> \> \textbf{do}\\
\textbf{\scriptsize4.} \> \> \> \> \textbf{for} (all combinations of two routes in $Y$) \textbf{do}\\
\textbf{\scriptsize5.} \> \> \> \> \> \> Do 1-1 vertex exchange;\\
\textbf{\scriptsize6.} \> \> \> \> \> \> Do 1-0 vertex exchange;\\
\textbf{\scriptsize7.} \> \> \> \> \> \> Do 2-1 vertex exchange;\\
\textbf{\scriptsize8.} \> \> \> \> \textbf{end-for};\\
\textbf{\scriptsize9.} \> \> \textbf{while} (did any inter-route improvement);\\
\textbf{\scriptsize10.} \> \> \emph{// intra-route improvements}\\
\textbf{\scriptsize11.} \> \> \textbf{for} (each route in $Y$) \textbf{do}\\
\textbf{\scriptsize12.} \> \> \> \> Do 3-opt improvement;\\
\textbf{\scriptsize13.} \> \> \textbf{end-for};\\
\textbf{\scriptsize14.} \textbf{while} (did any intra-route improvement); \\
\textbf{\scriptsize15.} \textbf{return} {$Y$};
\end{tabbing}
 }
\end{minipage}
}
}
\end{center}
\caption{Description of the sequence of the inter and intra-route local searches.}
\label{fig_inter_intra_route_improvements}
\end{figure}
At each iteration of the algorithm of Figure~\ref{fig_inter_intra_route_improvements}, the feasible solution available is first subject to inter-route improvements, i.e., procedures that exchange vertices between different routes (lines 2-9, Figure~\ref{fig_inter_intra_route_improvements}). Precisely, for all combinations of two routes, three kinds of vertex exchanges are performed: (i) 1-1, where a vertex from a route is exchanged with a vertex from another route, (ii) 1-0, where a vertex from a route is moved to another one and (iii) 2-1, where two adjacent vertices from a route are exchanged with a vertex from another route. In the three cases, given a pair of routes, an exchange is only allowed if it preserves the solution's feasibility and decreases the total sum of the routes' times. At a call of any of the vertex exchange procedures, the algorithm only performs a single exchange: the first possible by analyzing the routes from beginning to end. 

The sequence of inter-route improvements described above is performed until no more exchanges are possible. Then, the algorithm performs the intra-route improvements (lines 10-13, Figure~\ref{fig_inter_intra_route_improvements}). Precisely, for each route of the solution, the classical 3-opt operator \citep{Lin1965} is applied.
If any improvement is achieved through the 3-opt operator, the algorithm resumes the main loop by performing inter-route improvements once again. Otherwise, it returns the current solution and terminates (line 15, Figure~\ref{fig_inter_intra_route_improvements}).
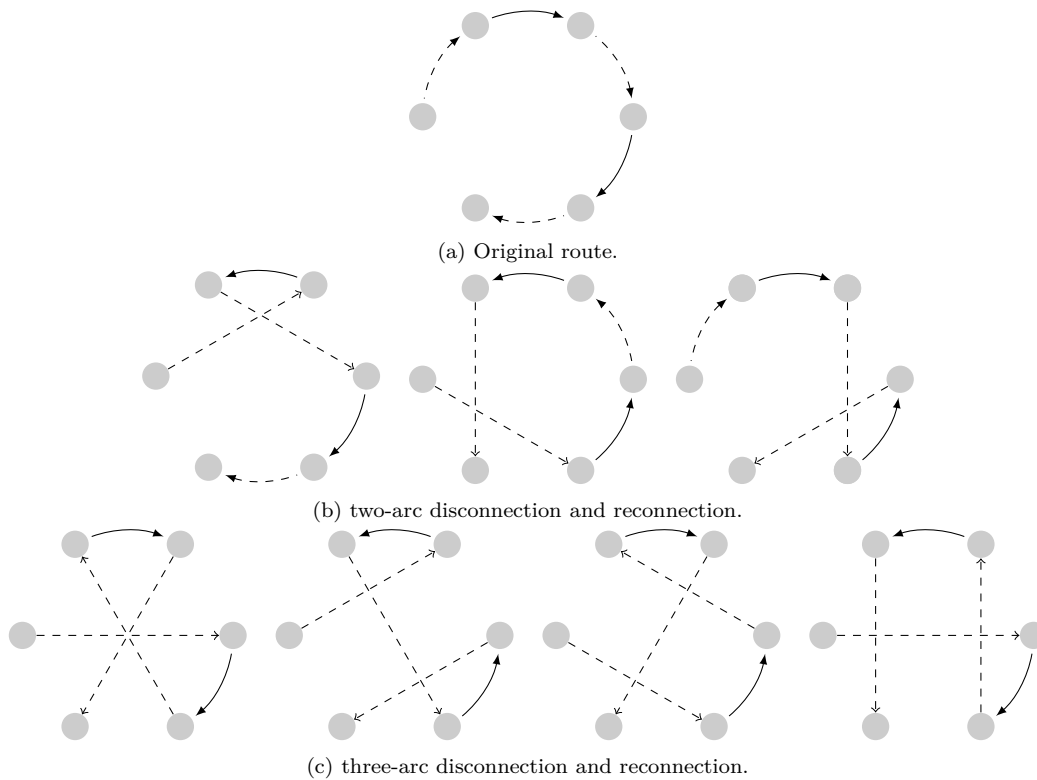
\begin{figure}
\centering
\begin{subfigure}{0.25\textwidth}
\centering
\begin{tikzpicture}
\def \n {6}
\def \radius {1.4cm}
\def \margin {10} 

\tikzstyle{state}=[
            circle,
            fill = gray!40,
            minimum size = 1mm
        ]

\foreach \s in {1,...,\n}
{
  \node[state] (\s) at ({60 * (\s - 1)}:\radius) {};
}

\foreach \s in {1,3,5}
{
  \draw[dashed,<-, >=latex] ({60 * (\s - 1)+\margin}:\radius) 
    arc ({60 * (\s - 1)+\margin}:{60 * (\s)-\margin}:\radius);
}

\foreach \s in {2,6}
{
  \draw[<-, >=latex] ({60 * (\s - 1)+\margin}:\radius) 
    arc ({60 * (\s - 1)+\margin}:{60 * (\s)-\margin}:\radius);
}

\end{tikzpicture}
\caption{Original route.}
\label{fig:example:original}
\end{subfigure}\\%
\begin{subfigure}{\textwidth}
\centering
\begin{tikzpicture}
\def \n {6}
\def \radius {1.4cm}
\def \margin {10} 

\tikzstyle{state}=[
            circle,
            fill = gray!40,
            minimum size = 1mm
        ]

\foreach \s in {1,...,\n}
{
  \node[state] (\s) at ({60 * (\s - 1)}:\radius) {};
}

\foreach \s in {5}
{
  \draw[dashed,<-, >=latex] ({60 * (\s - 1)+\margin}:\radius) 
    arc ({60 * (\s - 1)+\margin}:{60 * (\s)-\margin}:\radius);
}

\path[dashed,->] (3) edge node {} (1);
\path[dashed,->] (4) edge node {} (2);

\foreach \s in {2}
{
  \draw[->, >=latex] ({60 * (\s - 1)+\margin}:\radius) 
    arc ({60 * (\s - 1)+\margin}:{60 * (\s)-\margin}:\radius);
}

\foreach \s in {6}
{
  \draw[<-, >=latex] ({60 * (\s - 1)+\margin}:\radius) 
    arc ({60 * (\s - 1)+\margin}:{60 * (\s)-\margin}:\radius);
}

\end{tikzpicture}%
\quad
\begin{tikzpicture}
\def \n {6}
\def \radius {1.4cm}
\def \margin {10} 

\tikzstyle{state}=[
            circle,
            fill = gray!40,
            minimum size = 1mm
        ]

\foreach \s in {1,...,\n}
{
  \node[state] (\s) at ({60 * (\s - 1)}:\radius) {};
}

\foreach \s in {1}
{
  \draw[dashed,->, >=latex] ({60 * (\s - 1)+\margin}:\radius) 
    arc ({60 * (\s - 1)+\margin}:{60 * (\s)-\margin}:\radius);
}

\foreach \s in {2,6}
{
  \draw[->, >=latex] ({60 * (\s - 1)+\margin}:\radius) 
    arc ({60 * (\s - 1)+\margin}:{60 * (\s)-\margin}:\radius);
}

\path[dashed,->] (4) edge node {} (6);
\path[dashed,->] (3) edge node {} (5);

\end{tikzpicture}%
\quad
\begin{tikzpicture}
\def \n {6}
\def \radius {1.4cm}
\def \margin {10} 

\tikzstyle{state}=[
            circle,
            fill = gray!40,
            minimum size = 1mm
        ]

\foreach \s in {1,...,\n}
{
  \node[state] at ({60 * (\s - 1)}:\radius) {};
}

\foreach \s in {3}
{
  \node[state] at ({60 * (\s - 1)}:\radius) {};
  \draw[dashed,<-, >=latex] ({60 * (\s - 1)+\margin}:\radius) 
    arc ({60 * (\s - 1)+\margin}:{60 * (\s)-\margin}:\radius);
}

\foreach \s in {2}
{
  \node[state] at ({60 * (\s - 1)}:\radius) {};
  \draw[<-, >=latex] ({60 * (\s - 1)+\margin}:\radius) 
    arc ({60 * (\s - 1)+\margin}:{60 * (\s)-\margin}:\radius);
}

\path[dashed,->] (2) edge node {} (6);
\path[dashed,->] (1) edge node {} (5);

\foreach \s in {6}
{
  \node[state] at ({60 * (\s - 1)}:\radius) {};
  \draw[->, >=latex] ({60 * (\s - 1)+\margin}:\radius) 
    arc ({60 * (\s - 1)+\margin}:{60 * (\s)-\margin}:\radius);
}

\end{tikzpicture}
\caption{two-arc disconnection and reconnection.}
\label{fig:example:2opt}
\end{subfigure}\\%
\begin{subfigure}{\textwidth}
\centering
\begin{tikzpicture}
\def \n {6}
\def \radius {1.4cm}
\def \margin {10} 

\tikzstyle{state}=[
            circle,
            fill = gray!40,
            minimum size = 1mm
        ]

\foreach \s in {1,...,\n}
{
  \node[state] (\s) at ({60 * (\s - 1)}:\radius) {};
}

\foreach \s in {2,6}
{
  \draw[<-, >=latex] ({60 * (\s - 1)+\margin}:\radius) 
    arc ({60 * (\s - 1)+\margin}:{60 * (\s)-\margin}:\radius);
}

\path[dashed,->] (4) edge node {} (1);
\path[dashed,->] (6) edge node {} (3);
\path[dashed,->] (2) edge node {} (5);

\end{tikzpicture}%
\quad
\begin{tikzpicture}
\def \n {6}
\def \radius {1.4cm}
\def \margin {10} 

\tikzstyle{state}=[
            circle,
            fill = gray!40,
            minimum size = 1mm
        ]

\foreach \s in {1,...,\n}
{
  \node[state] (\s) at ({60 * (\s - 1)}:\radius) {};
}

\foreach \s in {2,6}
{
  \draw[->, >=latex] ({60 * (\s - 1)+\margin}:\radius) 
    arc ({60 * (\s - 1)+\margin}:{60 * (\s)-\margin}:\radius);
}

\path[dashed,->] (4) edge node {} (2);
\path[dashed,->] (3) edge node {} (6);
\path[dashed,->] (1) edge node {} (5);

\end{tikzpicture}%
\quad
\begin{tikzpicture}
\def \n {6}
\def \radius {1.4cm}
\def \margin {10} 

\tikzstyle{state}=[
            circle,
            fill = gray!40,
            minimum size = 1mm
        ]

\foreach \s in {1,...,\n}
{
  \node[state] (\s) at ({60 * (\s - 1)}:\radius) {};
}

\foreach \s in {2}
{
  \draw[<-, >=latex] ({60 * (\s - 1)+\margin}:\radius) 
    arc ({60 * (\s - 1)+\margin}:{60 * (\s)-\margin}:\radius);
}

\path[dashed,->] (4) edge node {} (6);
\path[dashed,->] (1) edge node {} (3);
\path[dashed,->] (2) edge node {} (5);

\foreach \s in {6}
{
  \draw[->, >=latex] ({60 * (\s - 1)+\margin}:\radius) 
    arc ({60 * (\s - 1)+\margin}:{60 * (\s)-\margin}:\radius);
}

\end{tikzpicture}%
\quad
\begin{tikzpicture}
\def \n {6}
\def \radius {1.4cm}
\def \margin {10} 

\tikzstyle{state}=[
            circle,
            fill = gray!40,
            minimum size = 1mm
        ]

\foreach \s in {1,...,\n}
{
  \node[state] (\s) at ({60 * (\s - 1)}:\radius) {};
}

\foreach \s in {6}
{
  \draw[<-, >=latex] ({60 * (\s - 1)+\margin}:\radius) 
    arc ({60 * (\s - 1)+\margin}:{60 * (\s)-\margin}:\radius);
}

\path[dashed,->] (4) edge node {} (1);
\path[dashed,->] (6) edge node {} (2);
\path[dashed,->] (3) edge node {} (5);

\foreach \s in {2}
{
  \draw[->, >=latex] ({60 * (\s - 1)+\margin}:\radius) 
    arc ({60 * (\s - 1)+\margin}:{60 * (\s)-\margin}:\radius);
}

\end{tikzpicture}
\caption{three-arc disconnection and reconnection.}
\label{fig:example:3opt}
\end{subfigure}
\caption{Example of a round of the 3-opt operator on an arbitrary route of a directed graph. The dashed arcs in the original route of (a) are the candidates for disconnection. The arc rearrangements of (b) and (c) disconnect two and three of the original arcs, respectively. Notice that, in some cases of reconnection, some arcs of the original route are preserved and others are reversed.}
\label{fig:example:kopt}
\end{figure}

Without loss of generality, a $k$-opt operator works by repeatedly disconnecting a given route in up to $k$ places and, then, testing for improvement all the possible routes obtained from reconnecting the initially damaged route in all the feasible manners. The operator terminates when no more improvements are possible from removing (and repairing) any combination of $k$ or less arcs. At this point, the route is called $k$-optimal. 
In Figure~\ref{fig:example:kopt}, we give an example of a round of the 3-opt operator on an arbitrary route of a directed graph. Accordingly, the original route of Figure~\ref{fig:example:original} is disconnected in three places --- identified by dashed arcs ---, and reconnected within seven possible manners (Figures~\ref{fig:example:2opt} and~\ref{fig:example:3opt}). The three first ones (Figure~\ref{fig:example:2opt}) are also the moves of the 2-opt operator, as one of the initially removed arcs is always restored (or reversed). In the latter cases (Figure~\ref{fig:example:3opt}), three of the original arcs are actually discarded. Notice that, depending on the type of reconnection, some arcs of the original route have to be reversed. Thus, in the case of not necessarily complete graphs (such as STOP), some of the rearrangements might not be always feasible.

\subsubsection{Vertex replacements}
\label{s_vertex_replacements}
Given a feasible solution, the vertex replacement procedure (Figure~\ref{fig_vertex_replacements}) works by replacing visited vertices with currently unvisited ones. The selection of unvisited vertices to be inserted in the solution is done exactly as in the insertion procedure described in Section~\ref{s_insertions}, i.e., it follows an either non-decreasing or non-increasing order (chosen at random) of vertex profits.
The procedure considers two types of replacements, namely 1-1 and 2-1 \emph{unvisited vertex exchanges}. The former replaces a visited vertex with an unvisited one, while, in the latter, two visited vertices are replaced by an unvisited one. A replacement is only allowed if it preserves the solution's feasibility and either (i) increases the profit sum or (ii) decreases the sum of the routes' times while maintaining the same profit sum.
At a call of any of the two types of replacements, the algorithm performs a single replacement: the first feasible one found by analyzing the routes from beginning to end.

\begin{figure}[!ht]
\begin{center}
\scalebox{1}
{
\framebox
{
\begin{minipage}[t]{25cm}
{\small
\begin{tabbing}
xxx\=xxx\=xxx\=xxx\=xxx\=xxx\=xxx\=xxx\=xxx\=xxx\= \kill
\textbf{Input: }{An initial feasible solution $Y$.} \\
\textbf{Output: }{A possibly improved version of $Y$.} \\
\textbf{\scriptsize1.} \textbf{do}\\
\textbf{\scriptsize2.} \> \> Do 1-1 unvisited vertex exchange;\\
\textbf{\scriptsize3.} \> \> Do 2-1 unvisited vertex exchange;\\
\textbf{\scriptsize4.} \textbf{while} (did any improvement); \\
\textbf{\scriptsize5.} \textbf{return} {$Y$};
\end{tabbing}
 }
\end{minipage}
}
}
\end{center}
\caption{Description of the sequence of vertex replacements.}
\label{fig_vertex_replacements}
\end{figure}

\subsection{Inter-route shifting perturbation}
\label{s_shifting}
The inter-route shifting perturbation algorithm works by individually testing if the visited vertices can be moved to any other route, as described in Figure~\ref{fig_shifting}. Initially, a copy $Y'$ of the initial solution $Y$ is done (line 1, Figure~\ref{fig_shifting}). Then, $Y$ is only used as a reference, while all the moves are performed in $Y'$, as to avoid multiple moves of a same vertex.
For each route, the algorithm attempts to move each of its vertices (either mandatory or profitable) to a different route in that solution, such that the destination route has the least possible increase in its time duration (lines 3-5, Figure~\ref{fig_shifting}). A vertex can be moved to any position of the destination route, as far as both the origin and the destination routes remain feasible. In this case, moves that increase the total sum of the routes' durations are also allowed.

\begin{figure}[!ht]
\begin{center}
\scalebox{1}
{
\framebox
{
\begin{minipage}[t]{25cm}
{\small
\begin{tabbing}
xxx\=xxx\=xxx\=xxx\=xxx\=xxx\=xxx\=xxx\=xxx\=xxx\= \kill
\textbf{Input: }{An initial feasible solution $Y$.} \\
\textbf{\scriptsize1.} Create a copy $Y'$ of $Y$;\\
\textbf{\scriptsize2.} \textbf{for} {(each route $r_1$ in $Y$)} \textbf{do}\\
\textbf{\scriptsize3.} \> \> \textbf{for} {(each vertex $i$ in $r_1$)} \textbf{do} \\
\textbf{\scriptsize4.} \> \> \> \> In $Y'$, try to move $i$ to a route $r_2$, $r_2 \neq r_1$, with least time increase;\\
\textbf{\scriptsize5.} \> \> \textbf{end-for}; \\
\textbf{\scriptsize6.} \> \> \textbf{if} {(any move from $r_1$ was done)} \textbf{then} \\
\textbf{\scriptsize7.} \> \> \> \> Try to insert unvisited vertices in $r_1$ of $Y'$;\\
\textbf{\scriptsize8.} \> \> \textbf{end-if};\\
\textbf{\scriptsize9.} \textbf{end-for}; \\
\textbf{\scriptsize10.} Replace $Y$ with $Y'$;
\end{tabbing}
 }
\end{minipage}
}
}
\end{center}
\caption{Description of the inter-route shifting perturbation.}
\label{fig_shifting}
\end{figure}

If at least one vertex is relocated during the attempt to move vertices from a given route, the algorithm proceeds by trying to insert in that route currently unvisited vertices (lines 6-8, Figure~\ref{fig_shifting}). In this case, vertices are inserted according to the procedure described in Section~\ref{s_insertions}, but only in the route under consideration. After all the originally visited vertices are tested for relocation, $Y$ is replaced by the possibly modified $Y'$ (line 10, Figure~\ref{fig_shifting}).

\subsection{The PR procedure}
\label{s_pr}

The PR procedure works by exploring neighborhoods connecting an initial solution to the ones of a given set of feasible solutions. Here, this set corresponds to the pool of solutions $\Lambda$, which plays the role of a long-term memory. Precisely, at each iteration of the PR procedure, we compute the intermediate solutions --- namely, the ``path'' --- between the initial solution and a solution selected from $\Lambda$, as detailed in Figure~\ref{fig_pr}. The algorithm receives as input an initial solution $Y$, the current pool of solutions $\Lambda$, the capacity of $\Lambda$ (\emph{max\_pool\_size}) and a similarity limit $\epsilon_2 \in [0,1]$ used to determine which pairs of solutions are eligible to be analyzed.

\begin{figure}[!ht]
\begin{center}
\scalebox{1}
{
\framebox
{
\begin{minipage}[t]{25cm}
{\small
\begin{tabbing}
xxx\=xxx\=xxx\=xxx\=xxx\=xxx\=xxx\=xxx\=xxx\=xxx\= \kill
\textbf{Input: }{An initial feasible solution $Y$, the pool of solutions $\Lambda$, \emph{max\_pool\_size} $\in \mathbb{Z}^+$},\\ 
{\emph{max\_pool\_size} $\geq 1$ and a similarity limit $\epsilon_2$.} \\
\textbf{Output: }{The possibly updated pool $\Lambda$.} \\
\textbf{\scriptsize1.} Set $\emph{best\_solution} \leftarrow Y$;\\
\textbf{\scriptsize2.} \textbf{for} {(each solution $X \in \Lambda$)} \textbf{do}\\
\textbf{\scriptsize3.} \> \> \emph{// Checks similarity between $X$ and $Y$}\\ 
\textbf{\scriptsize4.} \> \> \textbf{if} {($(2\times n_{X\cap Y})/(n_X + n_Y) < \epsilon_2$)} \textbf{then} \\
\textbf{\scriptsize5.} \> \> \> \> $\emph{current\_solution} \leftarrow$ best solution in the ``path'' from $Y$ to $X$ (see Figure~\ref{fig_pr_iter});\\
\textbf{\scriptsize6.} \> \> \> \> \textbf{if} {(\emph{current\_solution} is better than \emph{best\_solution})} \textbf{then} \\
\textbf{\scriptsize7.} \> \> \> \> \> \> $\emph{best\_solution} \leftarrow \emph{current\_solution}$;\\
\textbf{\scriptsize8.} \> \> \> \> \textbf{end-if};\\
\textbf{\scriptsize9.} \> \> \> \> $\emph{current\_solution} \leftarrow$ best solution in the ``path'' from $X$ to $Y$ (see Figure~\ref{fig_pr_iter});\\
\textbf{\scriptsize10.} \> \> \> \> \textbf{if} {(\emph{current\_solution} is better than \emph{best\_solution})} \textbf{then} \\
\textbf{\scriptsize11.} \> \> \> \> \> \> $\emph{best\_solution} \leftarrow \emph{current\_solution}$;\\
\textbf{\scriptsize12.} \> \> \> \> \textbf{end-if};\\
\textbf{\scriptsize13.} \> \> \textbf{end-if};\\
\textbf{\scriptsize14.} \textbf{end-for}; \\
\textbf{\scriptsize15.} \textbf{if} {($\emph{best\_solution}\notin \Lambda$)} \textbf{and} {(\emph{best\_solution} is better than the worst solution in $\Lambda$)} \textbf{then} \\
\textbf{\scriptsize16.} \> \> \textbf{if} {($|\Lambda| < \emph{max\_pool\_size}$)} \textbf{then} {$\Lambda \leftarrow \Lambda \cup \{\emph{best\_solution}\}$};\\
\textbf{\scriptsize17.} \> \> \textbf{else} {Replace the worst solution in $\Lambda$ with $\emph{best\_solution}$};\\
\textbf{\scriptsize18.} \textbf{end-if};\\
\textbf{\scriptsize19.} \textbf{return} {$\Lambda$};
\end{tabbing}
 }
\end{minipage}
}
}
\end{center}
\caption{Description of the general PR procedure.}
\label{fig_pr}
\end{figure}

Initially, the variable that keeps the best solution found so far in the PR procedure (\emph{best\_solution}) is set to the initial solution $Y$ (line 1, Figure~\ref{fig_pr}). Then, for each solution $X \in \Lambda$, we compute its similarity with $Y$, which is given by $(2\times n_{X\cap Y})/(n_X + n_Y)$, where $n_X$ and $n_Y$ stand for the number of vertices visited in the solutions $X$ and $Y$, respectively, and $n_{X \cap Y}$ is the number of vertices common to both solutions. If the similarity metric is inferior to the input limit $\epsilon_2$, the procedure computes the ``path'' from $Y$ to $X$ and the one from $X$ to $Y$. In both cases, the best solution found in the corresponding ``path'' (kept in the variable \emph{current\_solution}) is compared with the best solution found so far in the whole PR procedure (\emph{best\_solution}). Then, if applicable (i.e., if the profit sum of \emph{current\_solution} is greater than that of \emph{best\_solution}), \emph{best\_solution} is updated to \emph{current\_solution}. 
The whole loop described above is summarized at lines 2-14, Figure~\ref{fig_pr}.
After that, the algorithm attempts to add \emph{best\_solution} to $\Lambda$ (lines 15-18, Figure~\ref{fig_pr}). Then, the possibly updated pool is returned and the procedure terminates (line 19, Figure~\ref{fig_pr}). 

\begin{figure}[!ht]
\begin{center}
\scalebox{1}
{
\framebox
{
\begin{minipage}[t]{25cm}
{\small
\begin{tabbing}
xxx\=xxx\=xxx\=xxx\=xxx\=xxx\=xxx\=xxx\=xxx\=xxx\= \kill
\textbf{Input: }{A starting solution $Y_s$ and a guiding one $Y_g$.}\\
\textbf{Output: }{The possibly improved solution.} \\
\textbf{\scriptsize1.} $\emph{best\_solution} \leftarrow \emph{current\_solution} \leftarrow Y_s$;\\
\textbf{\scriptsize2.} $\emph{vertices\_to\_add} \leftarrow$ vertices that are visited in $Y_g$ and not in $Y_s$;\\
\textbf{\scriptsize3.} Sort $\emph{vertices\_to\_add}$ in terms of vertex profits (non-increasing or non-decreasing order);\\
\textbf{\scriptsize4.} \textbf{while} {($\emph{vertices\_to\_add} \neq \emptyset$)} \textbf{do}\\
\textbf{\scriptsize5.} \> \> \textbf{while} {($\emph{vertices\_to\_add} \neq \emptyset$) \textbf{and} (there is a feasible route in \emph{current\_solution})} \textbf{do}\\
\textbf{\scriptsize6.} \> \> \> \> Remove a vertex $i$ from \emph{vertices\_to\_add}, according to the sorting;\\
\textbf{\scriptsize7.} \> \> \> \> Try to insert $i$ in a still feasible route of \emph{current\_solution}, allowing infeasibility;\\
\textbf{\scriptsize8.} \> \> \textbf{end-while};\\
\textbf{\scriptsize9.} \> \> \textbf{for} (each infeasible route $j$ in \emph{current\_solution}) \textbf{do}\\
\textbf{\scriptsize10.} \> \> \> \> Sequentially remove profitable vertices from $j$ to restore its feasibility;\\
\textbf{\scriptsize11.} \> \> \textbf{end-for};\\
\textbf{\scriptsize12.} \> \> Improve \emph{current\_solution} through local searches (Section~\ref{s_local_searches});\\
\textbf{\scriptsize13.} \> \> \textbf{if} {(\emph{current\_solution} is better than \emph{best\_solution})} \textbf{then} \\
\textbf{\scriptsize14.} \> \> \> \> $\emph{best\_solution} \leftarrow \emph{current\_solution}$;\\
\textbf{\scriptsize15.} \> \> \textbf{end-if};\\
\textbf{\scriptsize16.} \textbf{end-while};\\
\textbf{\scriptsize17.} \textbf{return} \emph{best\_solution};
\end{tabbing}
 }
\end{minipage}
}
}
\end{center}
\caption{Algorithm for computing the ``path'' between two solutions.}
\label{fig_pr_iter}
\end{figure}

Given two input feasible solutions --- a \emph{starting} one $Y_s$ and a \emph{guiding} one $Y_g$ ---, the procedure of computing the ``path'' between them is outlined in Figure~\ref{fig_pr_iter}. Initially, the variables that keep the current and the best solutions found so far in the ``path'' from $Y_s$ and $Y_g$ are both set to $Y_s$ (line 1, Figure~\ref{fig_pr_iter}). Then, the set of vertices eligible for insertion (namely \emph{vertices\_to\_add}) is defined as the vertices that are visited in $Y_g$ and do not belong to $Y_s$ (line 2, Figure~\ref{fig_pr_iter}). Notice that all the vertices in \emph{vertices\_to\_add} are profitable, since $Y_s$ and $Y_g$ are both feasible and, thus, visit all the mandatory vertices. After that, \emph{vertices\_to\_add} is sorted in non-increasing or non-decreasing order of vertex profits (line 3, Figure~\ref{fig_pr_iter}). This order is chosen at random at each call of the procedure. 

Then, while there are eligible vertices in \emph{vertices\_to\_add}, the procedure alternates between adding vertices to the current solution --- even when it leads to infeasible routes --- and removing vertices to restore feasibility. Precisely, at each iteration of the main loop of Figure~\ref{fig_pr_iter} (lines 4-16), the algorithm attempts to add vertices to \emph{current\_solution} as follows. First, the next vertex in the order established for \emph{vertices\_to\_add} is removed from the set. Then, if possible, such vertex is inserted in the route and position that increase the least the sum of the routes' time durations. In this case, an insertion that makes a route infeasible in terms of time limit is also allowed. Nevertheless, once a route becomes infeasible, it is no longer considered for further insertions. The rounds of insertion continue until either \emph{vertices\_to\_add} gets empty or all the routes of \emph{current\_solution} become infeasible (see lines 5-8, Figure~\ref{fig_pr_iter}). 

At this point, the algorithm removes profitable vertices from \emph{current\_solution} to restore its feasibility. In particular, for each infeasible route, profitable vertices are sequentially removed in non-decreasing order of its profits until all the routes become feasible again (lines 9-11, Figure~\ref{fig_pr_iter}). If the removal of a vertex would disconnect the route it belongs, then this vertex is preserved. In these cases, the next candidate vertex is considered for removal, and so on. 

After the removals, the algorithm attempts to improve the now feasible \emph{current\_solution} through the local searches described in Section~\ref{s_local_searches} (line 12, Figure~\ref{fig_pr_iter}), and, if applicable, \emph{best\_solution} is updated to \emph{current\_solution} (lines 13-15, Figure~\ref{fig_pr_iter}). After the main loop terminates, \emph{best\_solution} is returned (line 17, Figure~\ref{fig_pr_iter}).

\section{Implementation details}
\label{s_implementation_details}
All the codes were developed in C++, and the LP problems that arise in the cutting-plane used to reinforce formulation $\mathcal{F}$ were solved by means of the optimization solver ILOG CPLEX 12.6\footnote{http://www-01.ibm.com/software/commerce/optimization/cplex-optimizer/}. We kept the default configurations of CPLEX in our implementation.

In the sequel, we describe some specific implementation choices for the separation of cuts, along with the parameter configurations adopted in each algorithm. These configurations were established according to previous studies in the literature, as well as to pilot tests on a control set of 10 STOP instances, composed of both challenging instances and some of the smallest ones. This control set is detailed in Table~\ref{table_control_instances}, where we report, for each instance, the number of vertices ($|N|$), the number of vehicles ($|M|$) and the route duration limit ($T$). The reduced number of instances was chosen as a way to prevent the parameter configuration of the heuristic from being biased by the control set (overfitting). The whole benchmark of instances adopted in this study is later detailed in Section~\ref{s_benchmarks}.

\begin{table}[!ht]
\center
\caption{Control set of STOP instances used to tune the algorithms' parameters.}
\label{table_control_instances}
\begin{tabular}{lrrr}
\toprule
\textbf{Instance} & \multicolumn{1}{c}{$|N|$} & \multicolumn{1}{c}{$|M|$} & \multicolumn{1}{c}{$T$} \\
\midrule
p3.3.r\_5\% & 33 & 3 & 33.3 \\
p4.3.j\_5\% & 100 & 3 & 46.7 \\
p4.3.n\_5\% & 100 & 3 & 60.0 \\
p5.3.m\_5\% & 66 & 3 & 21.7 \\
p5.3.r\_5\% & 66 & 3 & 30.0 \\
p6.2.k\_5\% & 64 & 2 & 32.5 \\
p6.3.m\_5\% & 64 & 3 & 25.0 \\
p6.3.n\_5\% & 64 & 3 & 26.7 \\
p7.3.o\_5\% & 102 & 3 & 100.0 \\
p7.3.p\_5\% & 102 & 3 & 106.7 \\
\bottomrule
\end{tabular}
\end{table}

\subsection{Cut separation}
Regarding the separation of CCCs, we solved maximum flow sub-problems with the implementation of the preflow push-relabel algorithm of \cite{Goldberg1988} provided by the open-source Library for Efficient Modeling and Optimization in Networks --- LEMON~\citep{Dezso2011}. The knapsack sub-problems that arise during the separation of LCIs were solved through classical dynamic programming based on Bellman recursion \citep{Bellman57}.

While selecting cuts, we adopted the \emph{absolute violation} criterion to determine which inequalities are violated by a given solution and the so-called \emph{distance} criterion to properly compare two cuts, i.e., to determine which one is most violated by a solution. Given an $n$-dimensional column vector $w$ of binary variables, a point $\bar{w} \in \mathbb{R}^n$ and an inequality of the general form $a^T w \leq b$, with $a \in \mathbb{R}^n$, $b \in \mathbb{R}$, the absolute violation of this inequality with respect to $\bar{w}$ is simply given by $a^T \bar{w} - b$. Moreover, the distance from $a^T w \leq b$ to $\bar{w}$ corresponds to the Euclidean distance between the hyperplane $a^T w = b$ and $\bar{w}$, which is equal to $\frac{(a^T \bar{w} - b)}{\lVert a \rVert}$, where $\lVert a \rVert$ is the Euclidean norm of $a$.

We set two parameters for each type of inequality separated in the cutting-plane algorithm: a precision one used to classify the inequalities into violated or not (namely \emph{absolute violation precision}), and another one to discard cuts that are not sufficiently orthogonal to the most violated ones. The latter parameter determines the minimum angle that an arbitrary cut must form with the most violated cut, as not to be discarded. In practice, this parameter establishes the maximum acceptable inner product between the arbitrary cut and the most violated one. Accordingly, we call it the \emph{maximum inner product}. In the case of two inequalities $a_1^T w \leq b_1$ and $a_2^T w \leq b_2$, with $a_1,a_2 \in \mathbb{R}^n$ and $b_1,b_2 \in \mathbb{R}$, the inner product between them is given by $\frac{(a_1^Ta_2)}{\lVert a_1 \rVert \lVert a_2 \rVert}$ and corresponds to the cosine of the angle defined by them. Since AVICs are separated by complete enumeration, the aforementioned parameters do not apply to them. The values adopted for these parameters are shown in Table~\ref{table_parameters}. The tolerance input value $\epsilon_1$ of the cutting-plane algorithm was set to $10^{-3}$.
\begin{table}[!ht]
\center
\caption{Parameter configuration adopted in the separation and selection of valid inequalities in the cutting-plane used to reinforce $\mathcal{F}$.}
\label{table_parameters}
\begin{tabular}{llllll}
\toprule
& &  \multicolumn{2}{c}{\textbf{Parameter}} \\
\cmidrule{3-4}
\textbf{Inequalities} &  & Absolute violation precision & Maximum inner product\\
\midrule
CCCs & & 0.01 & 0.03\\
AVICs & & ~~-- & ~~-- \\
LCIs & & $10^{-5}$  & ~~--\\
\bottomrule
\end{tabular}
\end{table}

\subsection{Parameter configuration adopted for the FP and the LNS heuristics}
In Tables~\ref{table_fp_parameters} and~\ref{table_LNS_parameters}, we summarize the values adopted for the input parameters of the FP framework described in Figure~\ref{fig_ofp} and the LNS heuristic, respectively. In the latter table, we include the parameters related to the procedures called within the main algorithm described in Figure~\ref{fig_LNS}.

\begin{table}[!ht]
\center
\caption{Parameter configuration adopted for the FP heuristic.}
\label{table_fp_parameters}
\begin{tabular}
    {lllll
    }
\toprule
\textbf{Parameter} &  & \emph{max\_pumps} & $\lambda$ & $K$\\
\midrule
\textbf{Value} & & 2000 & 0.9 & 10\\
\bottomrule
\end{tabular}
\end{table}
%
\begin{table}[!ht]
\center
\caption{Parameter configuration adopted for the LNS heuristic.}
\label{table_LNS_parameters}
\begin{adjustbox}{width=1\textwidth}
\begin{tabular}
    {lcccccc
    }
\toprule
\textbf{Parameter} &  & \emph{max\_iter} & \emph{max\_pool\_size} & \emph{stalling\_limit} & \emph{removal\_percentage} & $\epsilon_2$\\
\midrule
\textbf{Value} & & \{1000,\,2000,\,5000\} & 20 & 100 & 0.75 & 0.9\\
\bottomrule
\end{tabular}
\end{adjustbox}
\end{table}

\section{Computational experiments}
\label{s_experiments}
The computational experiments were performed on a 64 bits Intel Core i7-4790K machine with 4.0 GHz and 15.0 GB of RAM, under Linux operating system. The machine has four physical cores, each one running at most two threads in hyper-threading mode. In our experiments, we tested several variations of the LNS heuristic obtained by considering FP and OFP, the original formulation $\mathcal{F}$ and its reinforced version, as well as different numbers of iterations for the main LNS algorithm. These variations are later detailed in Sections~\ref{s_results_FP} and~\ref{s_results_LNS}.

\subsection{Benchmark instances} 
\label{s_benchmarks}
In our experiments, we used the benchmark of instances adopted by~\cite{Assuncao2019}, which consists of complete graphs with up to 102 vertices. These instances were generated based on the TOP ones of \cite{Chao96} by randomly selecting vertices to be mandatory. In each instance, only a small percentage of the vertices (5\%) is set as mandatory, as to avoid infeasibility. The 387 instances in the benchmark are divided into seven data sets, according to the number of vertices of their graphs. Within a given data set, the instances only differ by the time limit imposed on the route duration and the number of vehicles. The characteristics of the seven data sets are detailed in table~\ref{table_instances}. For each set, it is reported the number of instances (\#), the number of vertices in the graphs ($|N|$), the possible numbers of vehicles available ($|M|$) and the range of values that the route duration limit $T$ assumes. 
\begin{table}[!ht]
\center
\caption{Description of the benchmark of STOP instances adopted in the experiments.}
\label{table_instances}
\begin{tabular}{cccccccc}
\toprule
\textbf{Set} & 1$\_5\%$ & 2$\_5\%$ & 3$\_5\%$ & 4$\_5\%$ & 5$\_5\%$ & 6$\_5\%$ & 7$\_5\%$\\
\midrule
\# & 54 & 33 & 60 & 60 & 78 & 42 & 60\\
$|N|$ & 32 & 21 & 33 & 100 & 66 & 64 & 102\\
$|M|$ & 2--4 & 2--4 & 2--4 & 2--4 & 2--4 & 2--4 & 2--4 \\
$T$  & 3.8--22.5 & 1.2--42.5 & 3.8--55 & 3.8--40 & 1.2--65 & 5--200 & 12.5--120 \\
\bottomrule
\end{tabular}
\end{table}

According to previous experiments~\citep{Assuncao2019}, instances whose graphs have greater dimensions are the hardest (sets 4\_5\%, 5\_5\% and 7\_5\%) to be solved to optimality. Moreover, within a same set, instances with greater route duration limits (given by $T$) also tend to be more difficult. This is possibly due to the fact that greater limits imply more feasible routes, thus increasing the search space.

In this work, we pre-processed all the instances by removing vertices and arcs that are inaccessible with respect to the limit $T$ imposed on the total traverse times of the routes. To this end, we considered the $R$ matrix defined in the beginning of Section~\ref{s_models}, which keeps, for each pair of vertices, the time duration of a minimum time path between them.
Moreover, we deleted all the arcs that either enter the origin $s$ or leave the destination $t$, as to implicitly satisfy constraints (\ref{b104}). Naturally, the time spent in these pre-processings are included in the execution times of the algorithms tested.

\subsection{Statistical analysis adopted}
Since all the heuristic algorithms proposed have randomized choices within their execution, we ran each algorithm 10 times for each instance to properly assess their performance. In this sense, we considered a unique set of 10 seeds common to all algorithms and instances tested. To evaluate the quality of the solutions obtained by the heuristics proposed, we compared them with the best known primal solutions/bounds provided by the cutting-plane algorithm of \cite{Assuncao2019}.

To assess the statistical significance of the results, we follow the tests suggested by \cite{Demsar2006} for the simultaneous comparison of multiple algorithms on different data (instance) sets. Precisely, we first apply the Iman-Davenport test~\citep{Iman80} to check the so-called \emph{null hypothesis}, i.e., the occurrence of equal performances with respect to a given indicator (e.g., quality of the solution's bounds). If the null hypothesis is rejected, and, thus, the algorithms' performances differ in a statistically significant way, a post-hoc test is performed to analyze these differences more closely. In our study, we adopt the post-hoc test proposed by \cite{Nemenyi63}.

The Iman-Davenport test is a more accurate version of the non-parametric Friedman test~\citep{Friedman37}. In both tests, the algorithms considered are ranked according to an indicator of performance. Let $I$ be the set of instances and $J$ be the set of algorithms considered. In our case, the algorithms are ranked for each instance separately, such that $r^j_i$ stands for the rank of an algorithm $j \in J$ while solving an instance $i \in I$. Accordingly, the value of $r^j_i$ lies within the interval $[1,|J|]$, such that better performances are linked to smaller ranks (in case of ties, average ranks are assigned).
The average ranks (over all instances) of the algorithms --- given by $R_j = \frac{1}{|I|}\sum\limits_{i \in I}{r^j_i}$ for all $j \in J$ --- are used to compute a statistic $F_F$, which follows the F distribution \citep{Sheskin2007}.


Critical values are determined considering a significance level $\alpha$, which, in this case, indicates the probability of the null hypothesis being erroneously rejected. In practical terms, the smaller $\alpha$, the greater the statistical confidence of the test. Accordingly, in the case of the Iman-Davenport test, the null hypothesis is rejected if the statistic $F_F$ is greater than the critical value. In our experiments, we alternatively test the null hypothesis by determining (through a statistical computing software) the so-called $p$-value, which provides the smallest level of significance at which the null hypothesis would be rejected. In other words, given an appropriate significance level $\alpha$ (usually, at most 5\%), we can safely discard the null hypothesis if $p$-value $\leq \alpha$.

Once the null hypothesis is rejected, we can apply the post-hoc test of \cite{Nemenyi63}, which compares the algorithms in a pairwise manner. The performances of two algorithms $j,k \in J$ are significantly different if the corresponding average ranks $R_j$ and $R_k$ differ by at least a Critical Difference (CD).
In our experiments, we used the R open software for statistical computing\footnote{https://www.r-project.org/} to compute all of the statistics needed, including the average ranks and the CDs.

\subsection{Bound improvement provided by the new valid inequalities}
We first analyzed the impact of the new inequalities proposed in Section~\ref{s_cuts} on the strength of the formulation $\mathcal{F}$. To this end, we computed the dual (upper) bounds obtained from adding these inequalities to $\mathcal{L}$ (the linear relaxation of $\mathcal{F}$) according to 10 different configurations, as described in Table~\ref{table_cuts_configurations}. To ease comparisons, we also display the results for previous inequalities proposed in the literature, namely the General Connectivity Constraints (GCCs) of \cite{Bianchessi2017} and the Conflict Cuts (CCs) of \cite{Assuncao2019}. For each instance and configuration, we solved the cutting-plane algorithm described in Section~\ref{s_cutting-plane} while considering only the types of inequalities of the corresponding configuration. For the configurations including GCCs and CCs, we used, within the cutting-plane algorithm, the separation procedures adopted by \cite{Assuncao2019}. 
In all these experiments, the CPLEX built-in cuts are disabled. 

\begin{table}[!ht]
\center
\caption{Configurations of valid inequalities. GCCs and CCs stand for the General Connectivity Constraints of \cite{Bianchessi2017} and the Conflict Cuts (CCs) of \cite{Assuncao2019}, respectively.}
\label{table_cuts_configurations}
\begin{tabular}{lccccc}
\toprule
& \multicolumn{5}{c}{Inequalities}\\
\cmidrule{2-6}
\textbf{Configuration} & GCCs & CCs & CCCs & LCIs & AVICs\\
\midrule
1 & $\times$ & & & &\\
2 & & $\times$ & & &\\
3 & & & $\times$ & &\\
4 & & & &  $\times$  & \\
5  & & & & &  $\times$ \\
6  & $\times$ & $\times$ & & &\\
7  & $\times$ & $\times$ & & $\times$ &\\
8  & $\times$ & $\times$ & & $\times$ & $\times$\\
9 & & & $\times$ & $\times$ &\\
10 & & & $\times$ & $\times$ & $\times$\\
\bottomrule
\end{tabular}
\end{table}

The results obtained are detailed in Table~\ref{table_results_lps_stop}.
The first column displays the name of each instance set. Then, for each configuration of inequalities, we give the average and the standard deviation (over all the instances in each set) of the percentage bound improvements obtained from the addition of the corresponding inequalities. Without loss of generality, given an instance, its percentage improvement in a configuration $i \in \{1,\dots,10\}$ is given by $100 \cdot \frac{UB_{LP} - UB_{i}}{UB_{LP}}$, where $UB_{LP}$ denotes the bound provided by $\mathcal{L}$, and $UB_{i}$ stands for the bound obtained from solving the cutting-plane algorithm of Section~\ref{s_cutting-plane} in the configuration $i$.
The last row displays the numerical results while considering the whole benchmark of instances.

\begin{landscape}
\begin{table}[ht]
\caption{Percentage dual (upper) bound improvements obtained from adding to $\mathcal{L}$ the inequalities of Section~\ref{s_cuts} according to the 10 configurations in Table~\ref{table_cuts_configurations}.}
\label{table_results_lps_stop}
\begin{subtable}{1\textwidth}
\setlength\tabcolsep{3pt}
\centering
    \begin{tabular}
    {l
    S[table-format=2.2]
    S[table-format=2.2]
    S[table-format=2.2]
    S[table-format=2.2]
    S[table-format=2.2]
    S[table-format=2.2]
    S[table-format=2.2]
    S[table-format=2.2]
    S[table-format=2.2]
    S[table-format=2.2]
    S[table-format=2.2]
    S[table-format=2.2]
    S[table-format=2.2]
    S[table-format=2.2]
    S[table-format=2.2]
    }
\toprule
& & \multicolumn{14}{c}{Configuration of inequalities}\\
\cmidrule{3-16}
 & & \multicolumn{2}{c}{1 --- GCCs} & & \multicolumn{2}{c}{2 --- CCs} & & \multicolumn{2}{c}{3 --- CCCs} & & \multicolumn{2}{c}{4 --- LCIs} & & \multicolumn{2}{c}{5 --- AVICs} \\ 
 \cmidrule{3-4} \cmidrule{6-7} \cmidrule{9-10} \cmidrule{12-13} \cmidrule{15-16}
\textbf{Set }& & {Avg (\%)} & {StDev (\%)} & & {Avg (\%)} & {StDev (\%)} & & {Avg (\%)} & {StDev (\%)} & & {Avg (\%)} & {StDev (\%)} & & {Avg (\%)} & {StDev (\%)} \\
\midrule
1\_5\% & & 7.32 & 4.67 & & 8.83 & 5.98 & & 8.88 & 4.09 & & 1.00 & 3.93 & & 5.05 & 3.41\\
2\_5\% & & 0.32 & 0.90 & & 0.86 & 2.44 & & 0.64 & 1.81 & & 0.51 & 1.43 & & 0.22 & 0.61\\
3\_5\% & & 1.48 & 1.35 & & 2.71 & 1.90 & & 3.79 & 2.60 & & 0.59 & 0.99 & & 1.39 & 1.14\\
4\_5\% & & 5.94 & 5.62 & & 5.15 & 6.46 & & 8.80 & 7.33 & & 0.00 & 0.01 & & 5.52 & 5.01\\
5\_5\% & & 0.18 & 0.65 & & 0.86 & 1.46 & & 1.10 & 2.01 & & 0.02 & 0.10 & & 0.22 & 0.43\\
6\_5\% & & 0.06 & 0.17 & & 0.69 & 2.24 & & 0.18 & 0.45 & & 0.03 & 0.07 & & 0.18 & 0.53\\
7\_5\% & & 5.96 & 4.24 & & 8.95 & 9.58 & & 9.69 & 4.77 & & 0.00 & 0.00 & & 7.30 & 5.97\\
\midrule
\textbf{Total}& & 3.18 & 4.46& & 4.03 & 5.95& & 5.01 & 5.59& & 0.26 & 1.52& & 3.00 & 4.31\\
\midrule
\end{tabular}
\end{subtable}

\bigskip
\begin{subtable}{1\textwidth}
\sisetup{table-format=2.2} 
\setlength\tabcolsep{3pt}
\centering
    \begin{tabular}
    {l
    S[table-format=2.2]
    S[table-format=2.2]
    S[table-format=2.2]
    S[table-format=2.2]
    S[table-format=2.2]
    S[table-format=2.2]
    S[table-format=2.2]
    S[table-format=2.2]
    S[table-format=2.2]
    S[table-format=2.2]
    S[table-format=2.2]
    S[table-format=2.2]
    S[table-format=2.2]
    S[table-format=2.2]
    S[table-format=2.2]
    }
& & \multicolumn{14}{c}{Configuration of inequalities}\\
\cmidrule{3-16}
 & & \multicolumn{2}{c}{6 --- GCCs \& CCs} & & \multicolumn{2}{c}{7 --- 6 \& LCIs} & & \multicolumn{2}{c}{8 --- 7 \& AVICs} & & \multicolumn{2}{c}{9 --- CCCs \& LCIs} & & \multicolumn{2}{c}{10 --- 9 \& AVICs}\\ 
 \cmidrule{3-4} \cmidrule{6-7} \cmidrule{9-10} \cmidrule{12-13} \cmidrule{15-16}
\textbf{Set }& & {Avg (\%)} & {StDev (\%)} & & {Avg (\%)} & {StDev (\%)} & & {Avg (\%)} & {StDev (\%)} & & {Avg (\%)} & {StDev (\%)} & & {Avg (\%)} & {StDev (\%)} \\
\midrule
1\_5\% & & 9.26 & 6.01 & & 9.46 & 6.00 & & 9.49 & 6.05 & & 9.77 & 6.65 & & 9.13 & 4.01\\
2\_5\% & & 0.85 & 2.41 & & 0.97 & 2.75 & & 1.00 & 2.82 & & 1.12 & 3.18 & & 1.21 & 3.41\\
3\_5\% & & 2.71 & 1.87 & & 3.22 & 2.05 & & 3.28 & 2.04 & & 4.00 & 2.73 & & 4.07 & 2.80\\
4\_5\% & & 7.27 & 6.62 & & 7.27 & 6.63 & & 7.52 & 6.58 & & 8.83 & 7.31 & & 10.00 & 9.77\\
5\_5\% & & 0.84 & 1.32 & & 0.87 & 1.40 & & 0.88 & 1.39 & & 1.08 & 1.97 & & 1.25 & 2.08\\
6\_5\% & & 0.70 & 2.24 & & 0.73 & 2.23 & & 0.74 & 2.24 & & 0.21 & 0.45 & & 0.28 & 0.55\\
7\_5\% & & 10.16 & 8.85 & & 10.08 & 8.59 & & 10.43 & 8.60 & & 9.55 & 4.70 & & 10.63 & 5.41\\
\midrule
\textbf{Total} & & 4.63 & 6.14& & 4.75 & 6.09& & 4.86 & 6.14& & 5.19 & 5.97& & 5.52 & 6.56\\
\bottomrule
\end{tabular}
\end{subtable}

\end{table}
\end{landscape}

The results exposed in Table~\ref{table_results_lps_stop} indicate that, on average, CCCs are the inequalities that strengthen formulation $\mathcal{F}$ the most, followed by CCs, GCCs, AVICs and LCIs. The results also show that coupling CCCs with LCIs and AVICs gives the best average bound improvement (5.52\%) among all the configurations of inequalities tested.

One may notice that, for some instance sets, coupling different types of inequalities gives worse average bound improvements than considering only a subset of them (see, e.g., set 5\_5\% under configuration 2 and 6, Table~\ref{table_results_lps_stop}). Also notice that, although CCCs dominate GCCs and CCs (Corollary~\ref{corol_dominance}), there are cases where CCCs give worse average bound improvements than CCs (see set 2\_5\%, Table~\ref{table_results_lps_stop}). Both behaviours can be explained by the fact that these inequalities are separated heuristically, as the separation procedures only consider the cuts that are violated by at least a constant factor. Moreover, they adopt a stopping condition based on bound improvement of subsequent iterations, which might halt the separation before all the violated cuts are found.
Then, in practical terms, the separation algorithms adopted do not necessarily give the actual theoretical bounds obtained from the addition of the inequalities proposed.

For completeness, we display, in Table~\ref{table_results_extra_stop}, the average number of cuts separated for each class of inequalities when considering configuration 10 --- the one adopted in our FP algorithms. We omitted the average number of AVICs, as they are separated by complete enumeration and the number of these cuts in each instance always corresponds to $2\times|E|$. 

\begin{table}
\caption{Average number of cuts separated by the cutting-plane algorithm running under configuration 10. AVICs were omitted, as they are separated by complete enumeration.}
\label{table_results_extra_stop}
\setlength\tabcolsep{2pt}
\centering
    \begin{tabular}
    {ll
    S[table-format=3.2]
    S[table-format=3.2]
    }
\toprule
\textbf{Set } & & {LCIs} & {CCCs} \\
\midrule
1\_5\% & & 1.78 & 13.26\\
2\_5\% & & 0.33 & 1.00\\
3\_5\% & & 3.12 & 10.80\\
4\_5\% & & 1.02 & 88.57\\
5\_5\% & & 0.67 & 17.41\\
6\_5\% & & 0.29 & 0.71\\
7\_5\% & & 0.38 & 68.28\\
\midrule
\textbf{Total} & & 1.14 & 31.51\\
\midrule
\end{tabular}
\end{table}

\subsection{Results for the FP algorithms}
\label{s_results_FP}
We first compared four variations of the FP framework discussed in Section~\ref{s_FP}, as summarized in Table~\ref{table_fp_variations}. Precisely, we considered both FP and OFP while based on the original formulation $\mathcal{F}$ and the one reinforced according to the cutting-plane algorithm described in Section~\ref{s_cutting-plane}. The latter formulation is referred to as $\mathcal{F}$+cuts.

\begin{table}[!ht]
\center
\caption{Variations of FP analyzed in our study. $\mathcal{F}$+cuts stands for the reinforced version of the formulation $\mathcal{F}$ discussed in Section~\ref{s_cutting-plane}.}
\label{table_fp_variations}
\begin{tabular}{lccccc}
\toprule
& \multicolumn{2}{c}{Framework} & & \multicolumn{2}{c}{Base formulation}\\
\cmidrule{2-3} \cmidrule{5-6}
\textbf{Algorithm} & FP & OFP & & $\mathcal{F}$ & $\mathcal{F}$+cuts\\
\midrule
{FP}\_{raw} & $\times$ & & & $\times$ & \\
{FP}\_{cuts} & $\times$ & & & & $\times$\\
{OFP}\_{raw} & & $\times$ & & $\times$ & \\
{OFP}\_{cuts}  & & $\times$ & & & $\times$\\
\bottomrule
\end{tabular}
\end{table}

In Table~\ref{table_detailed_results_fp}, we report the results obtained by the four FP algorithms described in Table~\ref{table_fp_variations}. For each algorithm and instance set, we report four values: (i) the average and (ii) the standard deviation of the relative gaps given by $100 \cdot \frac{LB^* - LB}{LB^*}$, where $LB^*$ is the best known lower (primal) bound for the instance (provided by the exact algorithm of \cite{Assuncao2019}), and $LB$ is the average bound (over the 10 executions) obtained by the corresponding FP algorithm; (iii) the average number of iterations/pumps (over the 10 executions and all instances of the set), and (iv) the average wall-clock processing time (in seconds).
We highlight that, for the cases where a feasible solution is not found in any of the 10 executions, the corresponding relative gaps of (i) are set to 100\%. The same is done in case no primal bound is available from the baseline exact algorithm.
Moreover, when the instance is proven to be infeasible, the gap is set to 0\%, as well as when $LB^* = LB = 0$.
Entries in bold discriminate the cases where the reinforcement of the original formulation led to better performances or not.

\begin{table}[!ht]
\caption{Summary of the results obtained by the four FP algorithms described in Table~\ref{table_fp_variations}. Bold entries highlight, for each instance set, the best algorithm(s) in terms of average gaps and number of pumps.}
\label{table_detailed_results_fp}
\begin{subtable}{1\textwidth}
\setlength\tabcolsep{3pt}
\centering
    \begin{tabular}
    {l
    S[table-format=2.2]
    S[table-format=2.2]
    S[table-format=2.2]
    S[table-format=2.2]
    S[table-format=2.2]
    S[table-format=2.2]
    S[table-format=2.2]
    S[table-format=2.2]
    S[table-format=2.2]
    S[table-format=2.2]
    }
\toprule
 & & \multicolumn{4}{c}{FP\_{raw}} & & \multicolumn{4}{c}{FP\_{cuts}} \\ 
 \cmidrule{3-6} \cmidrule{8-11}
  & & \multicolumn{2}{c}{Gap (\%)} & & & & \multicolumn{2}{c}{Gap (\%)} & &\\
  \cmidrule{3-4} \cmidrule{8-9}
\textbf{Set }& & {Avg} & {StDev} & {Pumps (\#)} & {Time (s)} & & {Avg} & {StDev} & {Pumps (\#)} & {Time (s)} \\
\midrule
1\_5\% & & 26.60 & 29.92 & 5.86 & 0.06 & & \textbf{26.19} & 29.58 & \textbf{4.59} & 0.10\\
2\_5\% & & 11.74 & 23.61 & 2.19 & 0.00 & & \textbf{9.18} & 20.82 & \textbf{1.75} & 0.00\\
3\_5\% & & 29.47 & 26.30 & 7.32 & 0.08 & & \textbf{28.24} & 24.33 & \textbf{6.12} & 0.14\\
4\_5\% & & 29.69 & 24.62 & 86.59 & 39.38 & & \textbf{18.49} & 19.70 & \textbf{53.41} & 56.76\\
5\_5\% & & 27.76 & 26.33 & 38.06 & 3.33 & & \textbf{24.25} & 25.67 & \textbf{34.91} & 10.96\\
6\_5\% & & \textbf{12.46} & 16.87 & 5.23 & 0.45 & & 13.66 & 18.07 & \textbf{5.05} & 0.70\\
7\_5\% & & 27.66 & 34.21 & 161.26 & 109.12 & & \textbf{16.42} & 24.44 & \textbf{53.40} & 83.76\\
\midrule
\textbf{Total} & & 25.12 & 27.51 & 48.80 & 23.77 & & \textbf{20.60} & 24.48 & \textbf{25.88} & 24.11\\
\midrule
\end{tabular}
\end{subtable}
\begin{subtable}{1\textwidth}
\setlength\tabcolsep{3pt}
\centering
    \begin{tabular}
    {l
    S[table-format=2.2]
    S[table-format=2.2]
    S[table-format=2.2]
    S[table-format=2.2]
    S[table-format=2.2]
    S[table-format=2.2]
    S[table-format=2.2]
    S[table-format=2.2]
    S[table-format=2.2]
    S[table-format=2.2]
    }
\toprule
 & & \multicolumn{4}{c}{OFP\_{raw}} & & \multicolumn{4}{c}{OFP\_{cuts}} \\ 
 \cmidrule{3-6} \cmidrule{8-11}
  & & \multicolumn{2}{c}{Gap (\%)} & & & & \multicolumn{2}{c}{Gap (\%)} & &\\
  \cmidrule{3-4} \cmidrule{8-9}
\textbf{Set }& & {Avg} & {StDev} & {Pumps (\#)} & {Time (s)} & & {Avg} & {StDev} & {Pumps (\#)} & {Time (s)} \\
\midrule
1\_5\% & & 25.46 & 29.64 & 19.74 & 0.30 & & \textbf{23.80} & 26.71 & \textbf{18.09} & 0.44\\
2\_5\% & & 11.54 & 23.70 & 9.14 & 0.01 & & \textbf{11.18} & 22.00 & \textbf{8.07} & 0.02\\
3\_5\% & & 27.30 & 23.63 & 25.65 & 0.38 & & \textbf{26.81} & 23.41 & \textbf{24.03} & 0.54\\
4\_5\% & & 28.46 & 24.36 & 102.70 & 61.43 & & \textbf{20.21} & 20.43 & \textbf{72.20} & 80.50\\
5\_5\% & & 24.48 & 25.10 & 65.49 & 7.46 & & \textbf{21.77} & 23.48 & \textbf{53.63} & 16.86\\
6\_5\% & & 13.50 & 21.02 & 66.72 & 7.77 & & \textbf{12.49} & 16.22 & \textbf{18.89} & 3.09\\
7\_5\% & & 25.20 & 31.66 & 187.30 & 133.95 & & \textbf{15.24} & 23.20 & \textbf{90.91} & 175.60\\
\midrule
\textbf{Total} & & 23.49 & 26.45 & 72.91 & 32.74 & & \textbf{19.67} & 23.06 & \textbf{45.08} & 43.59\\
\bottomrule
\end{tabular}
\end{subtable}
\end{table}
First, notice that, in terms of solution quality, the results of the four algorithms are rather poor, as the average relative gaps and standard deviations are quite high. Nevertheless, these results do not indicate unsuccessful performances, as the general FP framework was devised to find feasible solutions, usually at the expense of solution quality. As expected, the quality of the solutions obtained by the algorithms that use the OFP framework is slightly superior (on average) to that of solutions obtained by the algorithms based on pure FP. Nevertheless, as we will discuss later, we cannot assure that the gain of OFP is statistically significant in this case.

More importantly, the results suggest that, on average, the reinforcement of the original formulation $\mathcal{F}$ improves both the quality of the solutions obtained and the convergence (signaled by the number of pumps) of the algorithms to feasible solutions. In particular, for all instance sets but 6\_5\%, the average gaps of the solutions obtained by FP\_{cuts} are strictly smaller (and, thus, better) than those obtained by FP\_{raw}. In the case of the average numbers of pumps, FP\_{cuts} outperforms FP\_{raw} for all instance sets.
When comparing OFP\_{cuts} and OFP\_{raw}, the same pattern 
is verified. Regarding average execution times, the four algorithms are comparable.

We applied the Iman-Davenport and the Nemenyi tests to validate the statistical significance of the results discussed above. While comparing the relative gaps, we ranked the algorithms according to the values $100 \cdot \left(1 - \frac{LB^* - LB}{LB^*}\right)$, as to make smaller gaps imply greater ranks. Similarly, while comparing the number of pumps, the algorithms were ranked based on the maximum number of pumps (parameter \emph{max\_iter} of Figure~\ref{fig_ofp}) minus the actual number of pumps performed. With respect to the relative gaps, we obtained the statistic $F_F = 10.76$, with $p$-value = $5.62\cdot10^{-7}$. Regarding the number of pumps, we obtained the statistic $F_F = 94.06$, with $p$-value = $2.2\cdot10^{-16}$. Then, in both cases, we can safely reject the null hypothesis of equal performances.

To compare the performance of the algorithms in a pairwise manner, we proceeded with the Nemenyi test. Figure~\ref{fig:FP_ranks} depicts the average ranks of the four FP algorithms and the Critical Difference (CD) while considering a significance level $\alpha = 5\%$. Connections between algorithms indicate non-significant differences, i.e., the difference between the corresponding pair of average ranks is no greater than the CD. Figures~\ref{fig:FP_ranks_a} and~\ref{fig:FP_ranks_b} are based on the relative gaps and the number of pumps, respectively.

With respect to the relative gaps, we can conclude that the performance of OFP\_cuts is significantly better than those of OFP\_raw and FP\_raw. Moreover, FP\_cuts outperforms FP\_raw.
Nevertheless, we cannot conclude if OFP\_cuts is significantly better than FP\_cuts, neither that OFP\_raw is better than FP\_raw in this same indicator. Regarding the convergence (number of pumps), FP\_cuts outperforms all the other algorithms. In summary, from the results, we can safely conclude that the reinforcement of $\mathcal{F}$ here applied yields better performances both in terms of convergence and solution quality. In addition, considering both indicators, FP\_cuts stands out as the best option.
\begin{figure}[!ht]
\begin{subfigure}{\textwidth}
\centering
    \includegraphics[scale=0.7]{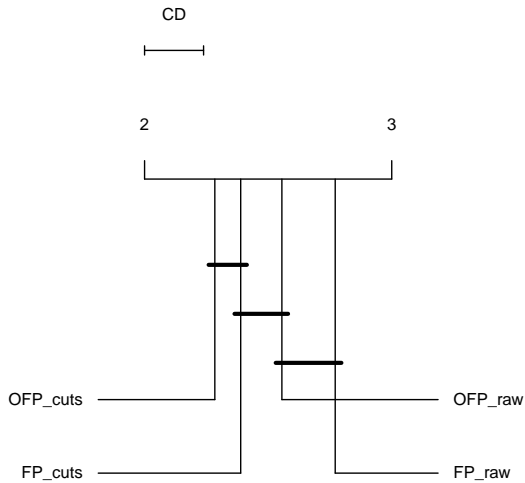}
    \caption{Ranks based on the relative gaps.}
  \label{fig:FP_ranks_a}
\end{subfigure}\\
\begin{subfigure}{\textwidth}
\centering
  \includegraphics[scale=0.7]{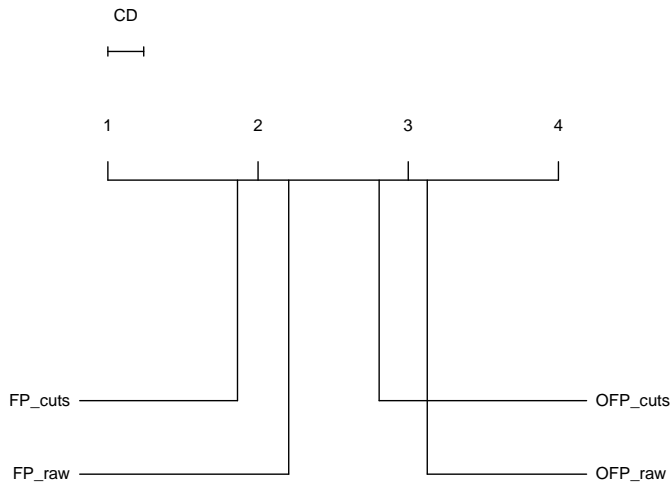}
  \caption{Ranks based on the number of pumps.}
  \label{fig:FP_ranks_b}
  \end{subfigure}
  \caption{FP algorithms' average ranks depicted on number lines, along with the Critical Difference (CD), when considering a significance level $\alpha = 5\%$. Connections between algorithms indicate non-significant differences.}
  \label{fig:FP_ranks}
\end{figure}



\subsection{Results for the LNS heuristic}
\label{s_results_LNS}
In the experiments concerning the LNS heuristic, we first tested the impact of the number of iterations on the quality of the final solutions obtained. To this end, we tested six variations of the heuristic by considering the two best FP algorithms --- FP\_cuts and OFP\_cuts --- coupled with the LNS framework of Section~\ref{s_LNS} running for \emph{max\_iter} $\in \{1000,2000,5000\}$ iterations. Hereafter, the variations of our heuristic are named (O)FP\_cuts\_LNS\_\emph{max\_iter}. Accordingly, without loss of generality, an algorithm that receives an initial solution provided by FP\_cuts and runs the LNS heuristic for 1000 iterations is referred to as FP\_cuts\_LNS\_1000.

In Table~\ref{table_results_LNS}, we summarize the results obtained by the six LNS algorithms tested. For each algorithm and instance set, we report five values: (i) the average and (ii) the standard deviation of the relative gaps given by $100 \cdot \frac{LB^* - LB}{LB^*}$, where $LB$ is the average bound (over the 10 executions) obtained by the corresponding algorithm. Recall that $LB^*$ is the best known bound for the instance (provided by the exact algorithm of \cite{Assuncao2019}); (iii) the average (over the 10 executions and all instances of the set) profit sum of the solutions obtained, and (iv) the average (over all instances of the set) profit sum of the best (in each round of 10 executions) solutions found. At last, we provide (v) the average wall-clock processing time (in seconds) of the execution of the LNS steps, excluding the time spent by the FP algorithm in finding the initial solution.
We highlight that, for the cases where a feasible solution is not provided in any of the 10 executions of the corresponding FP algorithm, the due relative gaps of (i) are set to 100\%. The same is done in case no primal bound is available from the baseline exact algorithm. Moreover, negative gaps indicate that the heuristic gives better primal bounds than the best known ones. When the instance is proven to be infeasible, the gap is set to 0\%, as well as when $LB^* = LB = 0$.
Entries in bold discriminate the cases where the OFP framework led to better performances than FP or otherwise.

\begin{table}[!ht]
\caption{Summary of the results obtained by six variations of the LNS algorithm tested. The execution times do not consider the time spent by the FP algorithms in finding initial solutions. Bold entries highlight, for each instance set, the best algorithm(s) in terms of average gaps and profit sums.}
\label{table_results_LNS}
\begin{subtable}{1\textwidth}
\setlength\tabcolsep{3pt}
\centering
    \begin{tabular}
    {ll
    S[table-format=2.2]
    S[table-format=2.2]
    S[table-format=3.2]
    S[table-format=3.2]
    S[table-format=2.2]
    l
    S[table-format=2.2]
    S[table-format=2.2]
    S[table-format=3.2]
    S[table-format=3.2]
    S[table-format=2.2]
    }
\toprule
 & & \multicolumn{5}{c}{FP\_{cuts}\_LNS\_1000} & & \multicolumn{5}{c}{OFP\_{cuts}\_LNS\_1000} \\ 
 \cmidrule{3-7} \cmidrule{9-13}
  & & \multicolumn{2}{c}{Gap (\%)} & \multicolumn{2}{c}{Profit sum} & & & \multicolumn{2}{c}{Gap (\%)} & \multicolumn{2}{c}{Profit sum} &\\
  \cmidrule{3-6} \cmidrule{9-12}
\textbf{Set }& & {Avg} & {StDev} & {Avg} & {Best} & {Time (s)} & & {Avg} & {StDev} & {Avg} & {Best} & {Time (s)}\\
\midrule
1\_5\% && 0.37 & 1.91 & 79.34 & \textbf{79.63} & 0.10 && \textbf{0.00} & 0.00 & \textbf{79.63} & \textbf{79.63} & 0.10\\
2\_5\% && \textbf{0.00} & 0.00 & \textbf{41.52} & \textbf{41.52} & 0.01 && \textbf{0.00} & 0.00 & \textbf{41.52} & \textbf{41.52} & 0.01\\
3\_5\% && \textbf{0.00} & 0.00 & \textbf{308.50} & \textbf{308.50} & 0.20 && \textbf{0.00} & 0.00 & \textbf{308.50} & \textbf{308.50} & 0.20\\
4\_5\% && -0.07 & 1.22 & 609.37 & 613.75 & 7.37 && \textbf{-0.19} & 1.08 & \textbf{612.15} & \textbf{614.23} & 7.10\\
5\_5\% && \textbf{1.34} & 11.32 & 602.36 & 603.14 & 1.04 && \textbf{1.34} & 11.32 & \textbf{602.40} & \textbf{603.21} & 1.04\\
6\_5\% && \textbf{0.01} & 0.03 & \textbf{375.20} & \textbf{375.29} & 0.55 && 0.02 & 0.05 & 375.11 & \textbf{375.29} & 0.55\\
7\_5\% && \textbf{1.79} & 12.90 & \textbf{306.32} & \textbf{311.22} & 1.42 && 3.48 & 18.08 & 292.82 & 295.20 & 1.36\\
\midrule
\textbf{Total} && \textbf{0.59} & 7.23 & \textbf{366.53} & \textbf{368.18} & 1.68 && 0.78 & 8.79 & 364.91 & 365.78 & 1.63\\
\midrule
\end{tabular}
\end{subtable}
\begin{subtable}{1\textwidth}
\setlength\tabcolsep{3pt}
\centering
    \begin{tabular}
    {ll
    S[table-format=2.2]
    S[table-format=2.2]
    S[table-format=3.2]
    S[table-format=3.2]
    S[table-format=2.2]
    l
    S[table-format=2.2]
    S[table-format=2.2]
    S[table-format=3.2]
    S[table-format=3.2]
    S[table-format=2.2]
    }
\toprule
 & & \multicolumn{5}{c}{FP\_{cuts}\_LNS\_2000} & & \multicolumn{5}{c}{OFP\_{cuts}\_LNS\_2000} \\ 
 \cmidrule{3-7} \cmidrule{9-13}
  & & \multicolumn{2}{c}{Gap (\%)} & \multicolumn{2}{c}{Profit sum} & & & \multicolumn{2}{c}{Gap (\%)} & \multicolumn{2}{c}{Profit sum} &\\
  \cmidrule{3-6} \cmidrule{9-12}
\textbf{Set }& & {Avg} & {StDev} & {Avg} & {Best} & {Time (s)} & & {Avg} & {StDev} & {Avg} & {Best} & {Time (s)}\\
\midrule
1\_5\% && 0.37 & 1.91 & 79.34 & \textbf{79.63} & 0.20 && \textbf{0.00} & 0.00 & \textbf{79.63} & \textbf{79.63} & 0.20\\
2\_5\% && \textbf{0.00} & 0.00 & \textbf{41.52} & \textbf{41.52} & 0.03 && \textbf{0.00} & 0.00 & \textbf{41.52} & \textbf{41.52} & 0.03\\
3\_5\% && \textbf{0.00} & 0.00 & \textbf{308.50} & \textbf{308.50} & 0.39 && \textbf{0.00} & 0.00 & \textbf{308.50} & \textbf{308.50} & 0.39\\
4\_5\% && -0.15 & 1.19 & 610.07 & 614.02 & 14.75 && \textbf{-0.25} & 1.11 & \textbf{612.78} & \textbf{614.30} & 14.35\\
5\_5\% && 1.31 & 11.32 & 602.70 & 603.14 & 2.08 && \textbf{1.30} & 11.32 & \textbf{602.74} & \textbf{603.21} & 2.09\\
6\_5\% && \textbf{0.00} & 0.02 & \textbf{375.24} & \textbf{375.29} & 1.09 && 0.01 & 0.03 & 375.21 & \textbf{375.29} & 1.11\\
7\_5\% && \textbf{1.74} & 12.90 & \textbf{306.73} & \textbf{311.28} & 3.04 && 3.44 & 18.08 & 293.13 & 295.25 & 2.85\\
\midrule
\textbf{Total} && \textbf{0.56} & 7.23 & \textbf{366.78} & \textbf{368.23} & 3.39 && 0.76 & 8.80 & 365.14 & 365.80 & 3.30\\
\midrule
\end{tabular}
\end{subtable}
\begin{subtable}{1\textwidth}
\setlength\tabcolsep{3pt}
\centering
    \begin{tabular}
    {ll
    S[table-format=2.2]
    S[table-format=2.2]
    S[table-format=3.2]
    S[table-format=3.2]
    S[table-format=2.2]
    l
    S[table-format=2.2]
    S[table-format=2.2]
    S[table-format=3.2]
    S[table-format=3.2]
    S[table-format=2.2]
    }
\toprule
 & & \multicolumn{5}{c}{FP\_{cuts}\_LNS\_5000} & & \multicolumn{5}{c}{OFP\_{cuts}\_LNS\_5000} \\ 
 \cmidrule{3-7} \cmidrule{9-13}
  & & \multicolumn{2}{c}{Gap (\%)} & \multicolumn{2}{c}{Profit sum} & & & \multicolumn{2}{c}{Gap (\%)} & \multicolumn{2}{c}{Profit sum} &\\
  \cmidrule{3-6} \cmidrule{9-12}
\textbf{Set }& & {Avg} & {StDev} & {Avg} & {Best} & {Time (s)} & & {Avg} & {StDev} & {Avg} & {Best} & {Time (s)}\\
\midrule
1\_5\% && 0.37 & 1.91 & 79.34 & \textbf{79.63} & 0.50 && \textbf{0.00} & 0.00 & \textbf{79.63} & \textbf{79.63} & 0.51\\
2\_5\% && \textbf{0.00} & 0.00 & \textbf{41.52} & \textbf{41.52} & 0.07 && \textbf{0.00} & 0.00 & \textbf{41.52} & \textbf{41.52} & 0.07\\
3\_5\% && \textbf{0.00} & 0.00 & \textbf{308.50} & \textbf{308.50} & 0.98 && \textbf{0.00} & 0.00 & \textbf{308.50} & \textbf{308.50} & 1.00\\
4\_5\% && -0.24 & 1.17 & 611.07 & \textbf{614.47} & 37.46 && \textbf{-0.30} & 1.14 & \textbf{613.28} & 614.43 & 35.98\\
5\_5\% && \textbf{1.29} & 11.32 & \textbf{602.96} & \textbf{603.21} & 5.30 && \textbf{1.29} & 11.32 & 602.94 & \textbf{603.21} & 5.30\\
6\_5\% && \textbf{0.00} & 0.01 & \textbf{375.26} & \textbf{375.29} & 2.74 && \textbf{0.00} & 0.02 & \textbf{375.26} & \textbf{375.29} & 2.76\\
7\_5\% && \textbf{1.72} & 12.90 & \textbf{306.92} & \textbf{311.30} & 7.50 && 3.40 & 18.09 & 293.45 & 295.27 & 7.15\\
\midrule
\textbf{Total} && \textbf{0.54} & 7.23 & \textbf{367.02} & \textbf{368.32} & 8.57 && 0.74 & 8.80 & 365.31 & 365.82 & 8.29\\
\bottomrule
\end{tabular}
\end{subtable}
\end{table}

The results indicate that, on average, the LNS algorithms that use the initial solutions provided by FP\_cuts and OFP\_cuts reach final solutions with comparable quality.
In addition, as expected, increasing the number of iterations does improve (on average) the quality of the solutions obtained. Also notice that the average execution times are almost the same for algorithms with a same number of iterations, which is in accordance with the stopping criterion adopted in the algorithms. In fact, the LNS heuristics run fairly fast, being the search for initial feasible solutions responsible for the majority of the computational effort (see again Table~\ref{table_detailed_results_fp}). Such behaviour was expected, since finding an initial solution for STOP is an NP-hard task (Corollary~\ref{corol_np_hard_feasibility}). 

Notice that, for the instance set with the greatest dimensions (set 7\_5\%), the standard deviation of the relative gaps obtained by the six algorithms were particularly high. This is partially due to the fact that, for a few instances in this set, the cutting-plane baseline could neither find feasible solutions nor prove their infeasibility. In these cases, the relative gaps were set to 100\%.

As for the FP heuristics, we also applied the Iman-Davenport and the Nemenyi tests to validate the statistical significance of the results discussed above.
We highlight that, although better (greater) profit sums imply smaller relative gaps on a per-instance basis, the same does not hold when we consider average values, since relative gaps are normalized by definition. Once our statistical tests use the results on a per-instance basis, considering either relative gaps or profit sums leads to a same ranking of the algorithms. In particular, we ranked the algorithms according to the values $100 \cdot \left(1 - \frac{LB^* - LB}{LB^*}\right)$, as to make smaller gaps imply greater ranks.
\begin{figure}[!ht]
\centering
    \includegraphics[scale=0.7]{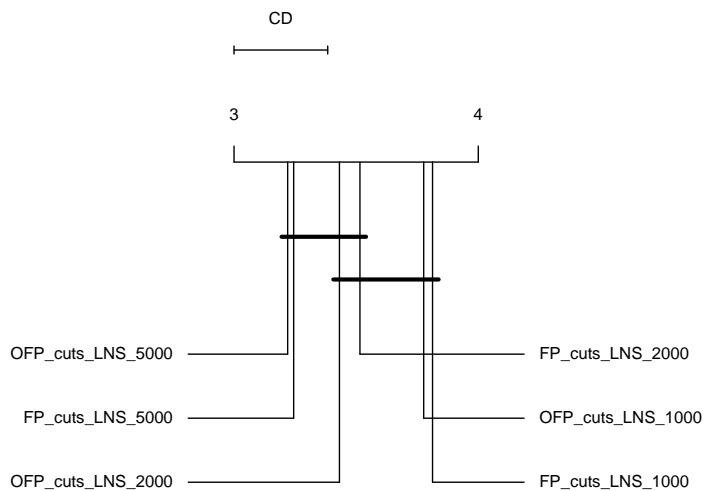}
  \caption{LNS algorithms' average ranks depicted on a number line, along with the Critical Difference (CD), when considering a significance level $\alpha = 5\%$. Connections between algorithms indicate non-significant differences.}
  \label{fig:LNS_ranks}
\end{figure}

Regarding the Iman-Davenport test, we obtained the statistic $F_F = 7.26$, with $p$-value = $9.49\cdot10^{-7}$, thus rejecting the null hypothesis of equal performances of all the six algorithms tested. Then, proceeding with the Nemenyi test, Figure~\ref{fig:LNS_ranks} depicts the average ranks of the six LNS algorithms and the Critical Difference (CD) while considering a significance level $\alpha = 5\%$. Connections between algorithms indicate non-significant differences with respect to CD.
From the results, we cannot conclude if the LNS algorithms based on FP\_cuts differ significantly (in terms of the quality of the solutions obtained) from the corresponding ones based on OFP\_cuts. On the other hand, the results clearly indicate that increasing the number of iterations from 1000 to 5000 does lead to a statistically significant improvement.

In Table~\ref{table_results_LNS_bounds}, we summarize the number of instances for which the two best variations of the heuristic (FP\_cuts\_LNS\_5000 and OFP\_cuts\_LNS\_5000) were able to reach the best previously known bounds. For each algorithm and instance set, we also indicate the number of cases where the LNS algorithm improved over the best known solutions.

\begin{table}[!ht]
\caption{Number of times the two best algorithms (FP\_cuts\_LNS\_5000 and OFP\_cuts\_LNS\_5000) reached and/or improved over the best known bounds. Bold entries highlight, for each instance set, the algorithm(s) that performed best.}
\label{table_results_LNS_bounds}
\centering
    \begin{tabular}
    {ll
    c
    c
    l
    c
    c
    }
\toprule
 && \multicolumn{2}{c}{FP\_{cuts}\_LNS\_5000} & & \multicolumn{2}{c}{OFP\_{cuts}\_LNS\_5000} \\ 
 \cmidrule{3-4} \cmidrule{6-7}
\textbf{Set } && \mbox{Reached best (\#)} & \mbox{Improved best (\#)} & & \mbox{Reached best (\#)} & \mbox{Improved best (\#)}\\
\midrule
1\_5\%& & \textbf{54/54} & 0/54& & \textbf{54/54} & 0/54\\
2\_5\%& & \textbf{33/33} & 0/33& & \textbf{33/33} & 0/33\\
3\_5\%& & \textbf{60/60} & 0/60& & \textbf{60/60} & 0/60\\
4\_5\%& & \textbf{59/60} & \textbf{18/60} & & \textbf{59/60} & 17/60\\
5\_5\%& & \textbf{76/78} & \textbf{3/78} & & \textbf{76/78} & \textbf{3/78}\\
6\_5\%& & \textbf{42/42} & 0/42& & \textbf{42/42} & 0/42\\
7\_5\%& & \textbf{58/60} & 0/60& & 54/60 & 0/60\\
\midrule
\textbf{Total} & & \textbf{382/387} & \textbf{21/387} & &378/387 & 20/387\\
\bottomrule
\end{tabular}
\end{table}

At last, we tested if dramatically increasing the number of iterations (to 100000) could close the gaps for two of the hardest instance sets (Sets 5\_5\% and 7\_5\%). The results, which are given in Table~\ref{table_results_LNS_100000_STOP}, suggest that this is not the case.

\begin{table}[!ht]
\caption{Summary of the results obtained by FP\_cuts\_LNS\_100000 and OFP\_cuts\_LNS\_100000 while solving two of the more challenging instance sets. The execution times do not consider the time spent by the FP algorithms in finding initial solutions. Bold entries highlight, for each instance set, the best algorithm(s) in terms of average gaps and profit sums.}
\label{table_results_LNS_100000_STOP}
\setlength\tabcolsep{3pt}
\centering
    \begin{tabular}
    {ll
    S[table-format=2.2]
    S[table-format=2.2]
    S[table-format=3.2]
    S[table-format=3.2]
    S[table-format=2.2]
    l
    S[table-format=2.2]
    S[table-format=2.2]
    S[table-format=3.2]
    S[table-format=3.2]
    S[table-format=2.2]
    }
\toprule
 & & \multicolumn{5}{c}{FP\_{cuts}\_LNS\_100000} & & \multicolumn{5}{c}{OFP\_{cuts}\_LNS\_100000} \\ 
 \cmidrule{3-7} \cmidrule{9-13}
  & & \multicolumn{2}{c}{Gap (\%)} & \multicolumn{2}{c}{Profit sum} & & & \multicolumn{2}{c}{Gap (\%)} & \multicolumn{2}{c}{Profit sum} &\\
  \cmidrule{3-6} \cmidrule{9-12}
\textbf{Set }& & {Avg} & {StDev} & {Avg} & {Best} & {Time (s)} & & {Avg} & {StDev} & {Avg} & {Best} & {Time (s)}\\
\midrule
5\_5\% && \textbf{1.27} & 11.32 & \textbf{603.12} & \textbf{603.27} & 103.01 && 1.28 & 11.32 & 603.08 & \textbf{603.27} & 104.04\\
7\_5\% && \textbf{1.69} & 12.91 & \textbf{307.14} & \textbf{311.33} & 148.69 && 3.37 & 18.10 & 293.73 & 295.35 & 143.86\\
\bottomrule
\end{tabular}
\end{table}

\subsubsection{Exploring the convergence of the LNS heuristic}
We performed additional experiments to have some insights on the convergence of the heuristic proposed. To this end, we tested stopping the execution whenever the heuristic converges to a local optimal. Precisely, instead of performing a fixed number of iterations, we halt the execution after \emph{stalling\_limit} iterations without improving --- in terms of profit sum --- the current best solution (see again the main algorithm in Figure~\ref{fig_LNS}). In this case, the PR procedure is not considered, once the heuristic stops right before the first call of PR.

As a way to also assess the influence of the initial solutions in this case, we considered the four variations of FP described in Table~\ref{table_fp_variations} as starting points for the LNS heuristic. The results are exposed in Table~\ref{table_results_conv_LNS}. Here, the columns have the same meaning as the ones of Table~\ref{table_results_LNS}. Then, we performed the Iman-Davenport test based on the relative gaps, obtaining the statistic $F_F = 1.76$, with $p$-value = $0.15$, which does not discard the null hypothesis of equal performances of all the four algorithms tested. 

\begin{table}[!ht]
\caption{Summary of the results obtained by the LNS heuristic when halting after \emph{stalling\_limit} iterations without profit sum improvement. The four FP algorithms --- described in Table~\ref{table_fp_variations} --- are tested as input for this version of the heuristic. The execution times do not consider the time spent by the FP algorithms in finding initial solutions. Bold entries highlight, for each instance set, the best algorithm(s) in terms of average gaps and profit sums.}
\label{table_results_conv_LNS}
\begin{subtable}{1\textwidth}
\setlength\tabcolsep{3pt}
\centering
    \begin{tabular}
    {ll
    S[table-format=2.2]
    S[table-format=2.2]
    S[table-format=3.2]
    S[table-format=3.2]
    S[table-format=2.2]
    l
    S[table-format=2.2]
    S[table-format=2.2]
    S[table-format=3.2]
    S[table-format=3.2]
    S[table-format=2.2]
    }
\toprule
 & & \multicolumn{5}{c}{FP\_{raw}\_LNS} & & \multicolumn{5}{c}{FP\_{cuts}\_LNS} \\ 
 \cmidrule{3-7} \cmidrule{9-13}
  & & \multicolumn{2}{c}{Gap (\%)} & \multicolumn{2}{c}{Profit sum} & & & \multicolumn{2}{c}{Gap (\%)} & \multicolumn{2}{c}{Profit sum} &\\
  \cmidrule{3-6} \cmidrule{9-12}
\textbf{Set }& & {Avg} & {StDev} & {Avg} & {Best} & {Time (s)} & & {Avg} & {StDev} & {Avg} & {Best} & {Time (s)}\\
\midrule
1\_5\% && 0.98 & 6.80 & 79.14 & \textbf{79.63} & 0.01 && \textbf{0.39} & 1.90 & \textbf{79.31} & \textbf{79.63} & 0.01\\
2\_5\% && \textbf{0.00} & 0.00 & \textbf{41.52} & \textbf{41.52} & 0.00 && \textbf{0.00} & 0.00 & \textbf{41.52} & \textbf{41.52} & 0.00\\
3\_5\% && \textbf{0.05} & 0.26 & \textbf{308.37} & \textbf{308.50} & 0.02 && 0.06 & 0.23 & 308.28 & \textbf{308.50} & 0.02\\
4\_5\% && 0.70 & 1.80 & 599.85 & \textbf{611.93} & 1.43 && \textbf{0.42} & 1.37 & \textbf{604.89} & 611.90 & 1.29\\
5\_5\% && \textbf{0.62} & 1.17 & \textbf{612.54} & \textbf{617.95} & 0.14 && 1.71 & 11.29 & 598.06 & 601.99 & 0.13\\
6\_5\% && \textbf{0.14} & 0.32 & \textbf{373.90} & \textbf{374.86} & 0.06 && 0.17 & 0.34 & 373.89 & \textbf{374.86} & 0.06\\
7\_5\% && 2.25 & 12.87 & 302.13 & 309.70 & 0.25 && \textbf{2.16} & 12.87 & \textbf{303.28} & \textbf{310.00} & 0.23\\
\midrule
\textbf{Total} && \textbf{0.74} & 5.74 & \textbf{366.27} & \textbf{370.60} & 0.30 && 0.83 & 7.23 & 364.32 & 367.42 & 0.27\\
\midrule
\end{tabular}
\end{subtable}
\begin{subtable}{1\textwidth}
\setlength\tabcolsep{3pt}
\centering
    \begin{tabular}
    {ll
    S[table-format=2.2]
    S[table-format=2.2]
    S[table-format=3.2]
    S[table-format=3.2]
    S[table-format=2.2]
    l
    S[table-format=2.2]
    S[table-format=2.2]
    S[table-format=3.2]
    S[table-format=3.2]
    S[table-format=2.2]
    }
\toprule
 & & \multicolumn{5}{c}{OFP\_{raw}\_LNS} & & \multicolumn{5}{c}{OFP\_{cuts}\_LNS} \\ 
 \cmidrule{3-7} \cmidrule{9-13}
  & & \multicolumn{2}{c}{Gap (\%)} & \multicolumn{2}{c}{Profit sum} & & & \multicolumn{2}{c}{Gap (\%)} & \multicolumn{2}{c}{Profit sum} &\\
  \cmidrule{3-6} \cmidrule{9-12}
\textbf{Set }& & {Avg} & {StDev} & {Avg} & {Best} & {Time (s)} & & {Avg} & {StDev} & {Avg} & {Best} & {Time (s)}\\
\midrule
1\_5\% && 1.01 & 7.10 & 79.12 & \textbf{79.63} & 0.01 && \textbf{0.02} & 0.08 & \textbf{79.60} & \textbf{79.63} & 0.01\\
2\_5\% && \textbf{0.00} & 0.00 & \textbf{41.52} & \textbf{41.52} & 0.00 && \textbf{0.00} & 0.00 & \textbf{41.52} & \textbf{41.52} & 0.00\\
3\_5\% && \textbf{0.02} & 0.09 & \textbf{308.42} & \textbf{308.50} & 0.02 && 0.07 & 0.32 & 308.25 & \textbf{308.50} & 0.02\\
4\_5\% && 0.66 & 1.80 & 602.13 & 611.38 & 1.45 && \textbf{0.27} & 1.20 & \textbf{607.86} & \textbf{612.58} & 1.33\\
5\_5\% && \textbf{0.56} & 0.72 & \textbf{609.26} & \textbf{618.01} & 0.14 && 1.73 & 11.29 & 598.07 & 602.37 & 0.14\\
6\_5\% && 2.49 & 15.41 & 350.09 & 351.00 & 0.06 && \textbf{0.14} & 0.34 & \textbf{374.03} & \textbf{374.86} & 0.06\\
7\_5\% && \textbf{2.38} & 12.93 & \textbf{297.35} & \textbf{309.83} & 0.28 && 3.76 & 18.04 & 290.60 & 294.43 & 0.22\\
\midrule
\textbf{Total} && \textbf{1.00} & 7.69 & 362.64 & \textbf{367.96} & 0.31 && \textbf{1.00} & 8.79 & \textbf{362.87} & 365.19 & 0.28\\
\midrule
\end{tabular}
\end{subtable}
\end{table}

From Table~\ref{table_results_conv_LNS}, one can notice that our heuristic obtains fairly good quality solutions for the four variations tested, with extremely small average execution times. Then, the results suggest that, although the reinforcement of the initial model makes FP converge to better quality solutions (as discussed in Section~\ref{s_results_FP}), the LNS itself is able to reach near-optimal solutions regardless of the initial ones used as input.

\subsection{Results for the original benchmark of TOP instances}
For completeness, the two variations of the heuristic that performed better for the STOP instances (FP\_cuts\_LNS\_5000 and OFP\_cuts\_LNS\_5000) were also run to solve the original benchmark of TOP instances~\citep{Chao96}. These instances are exactly the same STOP instances described in Section~\ref{s_benchmarks}, except for the lack of vertices set as mandatory. To evaluate the quality of the heuristic solutions obtained, we consider the best bounds found by the exact algorithm of \cite{Assuncao2019} (previous state-of-the-art), once the newest bounds obtained by \cite{Hanafi2020} are not reported in the publication. The results obtained are summarized in Tables~\ref{table_results_FP_TOP},~\ref{table_results_LNS_TOP} and~\ref{table_results_LNS_bounds_TOP}.

\begin{table}[!ht]
\caption{Summary of the results obtained by FP\_cuts and OFP\_cuts while solving the original benchmark of TOP instances. Bold entries highlight, for each instance set, the best algorithm(s) in terms of average gaps and number of pumps.}
\label{table_results_FP_TOP}
\setlength\tabcolsep{3pt}
\centering
     \begin{tabular}
    {l
    S[table-format=2.2]
    S[table-format=2.2]
    S[table-format=2.2]
    S[table-format=2.2]
    S[table-format=2.2]
    S[table-format=2.2]
    S[table-format=2.2]
    S[table-format=2.2]
    S[table-format=2.2]
    S[table-format=2.2]
    }
\toprule
 & & \multicolumn{4}{c}{FP\_{cuts}} & & \multicolumn{4}{c}{OFP\_{cuts}} \\ 
 \cmidrule{3-6} \cmidrule{8-11}
  & & \multicolumn{2}{c}{Gap (\%)} & & & & \multicolumn{2}{c}{Gap (\%)} & &\\
  \cmidrule{3-4} \cmidrule{8-9}
\textbf{Set }& & {Avg} & {StDev} & {Pumps (\#)} & {Time (s)} & & {Avg} & {StDev} & {Pumps (\#)} & {Time (s)} \\
\midrule
1 & & 27.91 & 26.17 & \textbf{4.18} & 0.09 & & \textbf{27.35} & 24.43 & 21.97 & 0.45\\
2 & & 35.96 & 31.49 & \textbf{3.49} & 0.01 & & \textbf{34.93} & 30.14 & 21.51 & 0.03\\
3 & & \textbf{32.16} & 23.69 & \textbf{5.75} & 0.14 & & 32.49 & 21.29 & 29.97 & 0.59\\
4 & & 24.72 & 13.85 & \textbf{14.31} & 27.57 & & \textbf{22.65} & 12.12 & 41.03 & 62.89\\
5 & & \textbf{26.60} & 22.23 & \textbf{9.15} & 2.68 & & 27.22 & 22.98 & 31.64 & 8.81\\
6 & & 17.06 & 19.64 & \textbf{4.78} & 0.57 & & \textbf{13.61} & 17.83 & 20.09 & 2.94\\
7 & & 31.50 & 21.01 & \textbf{13.46} & 59.77 & & \textbf{28.71} & 17.91 & 35.22 & 112.67\\
\midrule
\textbf{Total} & & 27.88 & 22.88 & \textbf{8.44} & 14.18 & & \textbf{26.76} & 21.68 & 29.93 & 29.47\\
\bottomrule
\end{tabular}
\end{table}

\begin{table}[!ht]
\caption{Summary of the results obtained by FP\_cuts\_LNS\_5000 and OFP\_cuts\_LNS\_5000 while solving the original benchmark of TOP instances. The execution times do not consider the time spent by the FP algorithms in finding initial solutions. Bold entries highlight, for each instance set, the best algorithm(s) in terms of average gaps and profit sums.}
\label{table_results_LNS_TOP}
\setlength\tabcolsep{3pt}
\centering
    \begin{tabular}
    {ll
    S[table-format=2.2]
    S[table-format=2.2]
    S[table-format=3.2]
    S[table-format=3.2]
    S[table-format=2.2]
    l
    S[table-format=2.2]
    S[table-format=2.2]
    S[table-format=3.2]
    S[table-format=3.2]
    S[table-format=2.2]
    }
\toprule
 & & \multicolumn{5}{c}{FP\_{cuts}\_LNS\_5000} & & \multicolumn{5}{c}{OFP\_{cuts}\_LNS\_5000} \\ 
 \cmidrule{3-7} \cmidrule{9-13}
  & & \multicolumn{2}{c}{Gap (\%)} & \multicolumn{2}{c}{Profit sum} & & & \multicolumn{2}{c}{Gap (\%)} & \multicolumn{2}{c}{Profit sum} &\\
  \cmidrule{3-6} \cmidrule{9-12}
\textbf{Set }& & {Avg} & {StDev} & {Avg} & {Best} & {Time (s)} & & {Avg} & {StDev} & {Avg} & {Best} & {Time (s)}\\
\midrule
1 && 1.67 & 9.66 & 111.07 & \textbf{112.04} & 0.65 && \textbf{0.00} & 0.00 & \textbf{112.04} & \textbf{112.04} & 0.65\\
2 && 6.67 & 15.14 & 131.45 & \textbf{140.45} & 0.18 && \textbf{0.61} & 2.42 & \textbf{139.14} & \textbf{140.45} & 0.19\\
3 && 3.17 & 13.96 & 411.22 & \textbf{414.67} & 1.18 && \textbf{0.00} & 0.00 & \textbf{414.67} & \textbf{414.67} & 1.18\\
4 && 0.34 & 3.11 & 801.22 & 804.23 & 38.20 && \textbf{-0.17} & 1.14 & \textbf{802.63} & \textbf{804.25} & 37.31\\
5 && 0.38 & 2.52 & 756.16 & 756.92 & 5.72 && \textbf{0.00} & 0.11 & \textbf{756.74} & \textbf{756.99} & 5.69\\
6 && 0.71 & 4.63 & 449.20 & \textbf{450.57} & 3.10 && \textbf{0.24} & 1.54 & \textbf{450.11} & \textbf{450.57} & 2.92\\
7 && 2.66 & 12.00 & 562.26 & 568.25 & 10.90 && \textbf{0.14} & 0.27 & \textbf{567.33} & \textbf{568.55} & 10.97\\
\midrule
\textbf{Total} && 1.91 & 9.59 & 503.01 & 506.14 & 9.39 && \textbf{0.07} & 1.00 & \textbf{505.56} & \textbf{506.21} & 9.24\\
\midrule
\end{tabular}
\end{table}

\begin{table}[!ht]
\caption{Number of times the two best algorithms (FP\_cuts\_LNS\_5000 and OFP\_cuts\_LNS\_5000) reached and/or improved over the primal bounds provided by the cutting-plane algorithm of~\cite{Assuncao2019} while solving the original benchmark of TOP instances. Bold entries highlight, for each instance set, the algorithm(s) that performed best.}
\label{table_results_LNS_bounds_TOP}
\centering
\begin{adjustbox}{width=1\textwidth}
    \begin{tabular}
    {ll
    c
    c
    l
    c
    c
    }
\toprule
 && \multicolumn{2}{c}{FP\_{cuts}\_LNS\_5000} & & \multicolumn{2}{c}{OFP\_{cuts}\_LNS\_5000} \\ 
 \cmidrule{3-4} \cmidrule{6-7}
\textbf{Set } && \mbox{Reached best (\#)} & \mbox{Improved best (\#)} & & \mbox{Reached best (\#)} & \mbox{Improved best (\#)}\\
\midrule
1& & \textbf{54/54} & 0/54& & \textbf{54/54} & 0/54\\
2& & \textbf{33/33} & 0/33& & \textbf{33/33} & 0/33\\
3& & \textbf{60/60} & 0/60& & \textbf{60/60} & 0/60\\
4& & \textbf{55/60} & \textbf{12/60} & & 53/60 & \textbf{12/60}\\
5& & \textbf{78/78} & 2/78& & \textbf{78/78} & \textbf{3/78}\\
6& & \textbf{42/42} & 0/42& & \textbf{42/42} & 0/42\\
7& & 56/60 & \textbf{1/60} & & \textbf{59/60} & \textbf{1/60}\\
\midrule
\textbf{Total} & &378/387 & 15/387 & & \textbf{379/387} & \textbf{16/387}\\
\midrule
\end{tabular}
\end{adjustbox}
\end{table}
In Table~\ref{table_results_FP_TOP}, we report, for each FP algorithm tested (FP\_cuts and OFP\_cuts) and instance set, four values: (i) the average and (ii) the standard deviation of the relative gaps given by $100 \cdot \frac{LB^* - LB}{LB^*}$. Recall that $LB^*$ is the best primal bound (for the instance) obtained by the exact algorithm of \cite{Assuncao2019}, and $LB$ is the average bound (over the 10 executions) obtained by the corresponding FP algorithm; (iii) the average number of iterations/pumps (over the 10 executions and all instances of the set), and (iv) the average wall-clock processing time (in seconds).
We highlight that, whenever $LB^* = LB = 0$, the corresponding gap of (i) is set to 0\%.

In Table~\ref{table_results_LNS_TOP}, for each LNS algorithm tested (FP\_cuts\_LNS\_5000 and OFP\_cuts\_LNS\_5000) and instance set, we report five values: (i) the average and (ii) the standard deviation of the relative gaps given by $100 \cdot \frac{LB^* - LB}{LB^*}$; (iii) the average (over the 10 executions and all instances of the set) profit sum of the solutions obtained, and (iv) the average (over all instances of the set) profit sum of the best (in each round of 10 executions) solutions found. At last, we provide (v) the average wall-clock processing time (in seconds) of the execution of the LNS steps, excluding the time spent by the FP algorithm in finding the initial solution. Also in this case, whenever $LB^* = LB = 0$, the corresponding gap of (i) is set to 0\%.

In Table~\ref{table_results_LNS_bounds_TOP}, we summarize the number of instances for which the two best variations of the heuristic (FP\_cuts\_LNS\_5000 and OFP\_cuts\_LNS\_5000) were able to reach the primal bounds obtained by baseline exact algorithm \citep{Assuncao2019}. For each algorithm and instance set, we also indicate the number of cases where the LNS algorithm improved over the best solutions provided by \cite{Assuncao2019}.

\section{Concluding remarks}
In this work, we prove that solely finding a feasible solution for the Steiner Team Orienteering Problem (STOP) is NP-hard and propose a Large Neighborhood Search (LNS) heuristic to solve the problem. The heuristic combines classical local searches from the literature of routing problems with a memory component based on Path Relinking (PR). Our heuristic is provided with initial solutions obtained by means of the matheuristic framework known as Feasibility Pump (FP). In our implementations, FP uses as backbone a commodity-based formulation reinforced by three classes of valid inequalities, being two of them also introduced in this work. 

We used the primal bounds provided by a state-of-the-art cutting-plane algorithm from the literature to guide our computational experiments, which showed the efficiency and effectiveness of the proposed heuristic in solving a benchmark of 387 instances. In particular, the heuristic could reach the best previously known bounds on 382 of these instances. Additionally, in 21 of these cases, our heuristic was even able to improve over the best known bounds. Overall, the heuristic solutions imply an average percentage gap of only 0.54\% when compared to the bounds of the cutting-plane baseline.

In terms of future work directions, we first remark that the two new classes of inequalities here proposed can be naturally extended for other routing problems. In addition, we believe that our heuristic might be also successful in solving closely related orienteering problems, such as OARP and TOARP, in which arcs (instead of vertices) are sub-divided into mandatory and profitable.

\section*{Acknowledgments}
This work was partially supported by the Brazilian National Council
for Scientific and Technological Development (CNPq), the Foundation
for Support of Research of the State of Minas Gerais, Brazil
(FAPEMIG), and Coordination for the Improvement of Higher Education
Personnel, Brazil (CAPES).

\bibliography{mybibfile}

\end{document}